\documentclass[10pt]{amsart}
\usepackage{mathrsfs}
\usepackage{amscd}
\usepackage{enumerate}
\usepackage[bb=fourier,cal=pxtx]{mathalfa}
\usepackage{amssymb}
\usepackage{mathtools}      
\usepackage{bm}             
\usepackage{indentfirst}    
\usepackage{amsthm}
\usepackage{thmtools}
\usepackage{enumitem}       
\usepackage[colorlinks=true,backref]{hyperref} 
\usepackage[font=small]{caption}

\renewcommand*{\backref}[1]{}
\renewcommand*{\backrefalt}[4]{\quad \tiny
  \ifcase #1 (\textbf{NOT CITED.})%
  \or    (Cited on page~#2.)%
  \else   (Cited on pages~#2.)%
  \fi}

\makeatletter

\def\MRbibitem{\@ifnextchar[\my@lbibitem\my@bibitem}

\def\mybiblabel#1#2{\@biblabel{{\hyperref{http://www.ams.org/mathscinet-getitem?mr=#1}{}{}{#2}}}}

\def\myhyperanchor#1{\Hy@raisedlink{\hyper@anchorstart{cite.#1}\hyper@anchorend}}

\def\my@lbibitem[#1]#2#3#4\par{%
  \item[\mybiblabel{#2}{#1}\myhyperanchor{#3}\hfill]#4%
  \@ifundefined{ifbackrefparscan}{}{\BR@backref{#3}}%
  \if@filesw{\let\protect\noexpand\immediate
    \write\@auxout{\string\bibcite{#3}{#1}}}\fi\ignorespaces%
}

\def\my@bibitem#1#2#3\par{%
  \refstepcounter\@listctr
  \item[\mybiblabel{#1}{\the\value\@listctr}\myhyperanchor{#2}\hfill]#3%
  \@ifundefined{ifbackrefparscan}{}{\BR@backref{#2}}%
  \if@filesw\immediate\write\@auxout
    {\string\bibcite{#2}{\the\value\@listctr}}\fi\ignorespaces%
}

\makeatother
\declaretheorem{theorem}

\declaretheorem[name=Acknowledgement, style=remark, numbered=no]{ack}

\hypersetup{bookmarksdepth = 3} 
\numberwithin{equation}{section}     

\setlist[enumerate,1]{label={\upshape(\alph*)},ref=\alph*}
\setlist[enumerate,2]{label={\upshape(\arabic*)},ref=\arabic*}

\newcommand{\R}{\mathbb{R}}
\newcommand{\Z}{\mathbb{Z}}
\newcommand{\N}{\mathbb{N}}

\newcommand{\cF}{\mathcal{F}}

\newcommand{\cL}{\mathcal{L}}
\newcommand{\cM}{\mathcal{M}}\newcommand{\cO}{\mathcal{O}}

\newcommand{\cU}{\mathcal{U}}

\newcommand{\cZ}{\mathcal{Z}}
\newcommand{\sA}{\mathscr{A}}

\newcommand{\C}{\mathbb{C}}

\renewcommand{\N}{\mathbb{N}}

\renewcommand{\R}{\mathbb{R}}

\renewcommand{\Z}{\mathbb{Z}}

\newtheorem{prop}{Proposition}[section]

\newtheorem{coro}[prop]{Corollary}
\newtheorem{lemm}[prop]{Lemma}
\newtheorem{fact}[prop]{Fact}
\newtheorem{prob}{Problem}
\newtheorem{claim}{Claim}

\theoremstyle{definition}

\theoremstyle{remark}
\newtheorem{rema}[prop]{Remark}

\renewcommand{\epsilon}{\varepsilon}

\renewcommand{\emptyset}{\varnothing}

\makeatletter
\DeclareFontFamily{OMX}{MnSymbolE}{}
\DeclareSymbolFont{MnLargeSymbols}{OMX}{MnSymbolE}{m}{n}
\SetSymbolFont{MnLargeSymbols}{bold}{OMX}{MnSymbolE}{b}{n}
\DeclareFontShape{OMX}{MnSymbolE}{m}{n}{
    <-6>  MnSymbolE5
   <6-7>  MnSymbolE6
   <7-8>  MnSymbolE7
   <8-9>  MnSymbolE8
   <9-10> MnSymbolE9
  <10-12> MnSymbolE10
  <12->   MnSymbolE12
}{}
\DeclareFontShape{OMX}{MnSymbolE}{b}{n}{
    <-6>  MnSymbolE-Bold5
   <6-7>  MnSymbolE-Bold6
   <7-8>  MnSymbolE-Bold7
   <8-9>  MnSymbolE-Bold8
   <9-10> MnSymbolE-Bold9
  <10-12> MnSymbolE-Bold10
  <12->   MnSymbolE-Bold12
}{}

\let\llangle\@undefined
\let\rrangle\@undefined
\DeclareMathDelimiter{\llangle}{\mathopen}%
                     {MnLargeSymbols}{'164}{MnLargeSymbols}{'164}
\DeclareMathDelimiter{\rrangle}{\mathclose}%
                     {MnLargeSymbols}{'171}{MnLargeSymbols}{'171}
\makeatother

\begin{document}
\title[A real version of Makarov and Smirnov's
formalism]{Thermodynamic formalism of interval maps for upper
semi-continuous potentials:~~~Makarov and Smirnov's formalism}
\date{\today}

\author{Yiwei Zhang}

\begin{thanks}
{Y.Z.\ was supported by project Fondecyt 3130622.}
\end{thanks}

\begin{abstract}
In this paper, we study the thermodynamic formalism of interval maps
$f$ with sufficient regularity, for a sub class $\mathcal{U}$
composed of upper semi-continuous potentials which includes both
H\"{o}lder and geometric potentials. We show that for a given $u\in
\mathcal{U}$ and negative values of $t$, the pressure function
$P(f,-tu)$ can be calculated in terms of the corresponding hidden
pressure function $\widetilde{P}(f,-tu)$. Determination of the
values $t\in(-\infty,0)$ at which $P(f,-tu)\neq
\widetilde{P}(f,-tu)$ is also characterized explicitly. When restricting to the H\"{o}lder continuous potentials, our result recovers Theorem B in \cite{LiRiv13b} for maps with non-flat critical points. While restricting to the geometric potentials, we develop a real version of
Makarov-Smirnov's formalism, in parallel
to the complex version shown in \cite[Theo A,B]{MakSmi00}. Moreover, our results also provide a simpler proof
(using \cite[Coro6.3]{Rue92}) of the original Makarov-Smirnov's
formalism in the complex setting, under an additional assumption
about non-exceptionality, i.e., \cite[Theo3.1]{MakSmi00}.

\end{abstract}

\maketitle

\section{Introduction}
\subsection{Thermodynamic formalism and Hyperbolicity}
Thermodynamic formalism can be interpreted as a set of ideas and
techniques derived from statistical mechanics. A systematic study of
thermodynamic formalism of uniform hyperbolic smooth dynamical
systems originates from the pioneer works of Sinai, Bowen and Ruelle
\cite{Bow75,Rue76,Sin72}, and has various applications from
dimension theory to number theory. One of the outstanding result is
that the uniform hyperbolicity implies the real analyticity of the
topological pressure function with H\"{o}lder continuous potentials,
and thus yields the absence of phase transitions.

Recently, the extension of thermodynamic formalism results to one
dimensional non-uniform hyperbolic \footnote{The non-uniform
hyperbolicity is usually caused by the existence of critical points
and neutral cycles.} dynamical systems has attracted great
interests. Various direct or indirect approaches have been proposed
to compensate the lack of uniform hyperbolicity in the dynamical
system. Perhaps one of the most promising and direct methods is via
the demonstration of the hyperbolicity in the potential. Suppose $X$
is a compact metric space, and a continuous map $T:X\to X$. A upper
semi-continuous potential $\phi:X\to \R\cup\{-\infty\}$ is said to
be \emph{hyperbolic}, if there is an integer $n\geq1$ such that the
function $S_{n}(\phi):=\sum_{j=0}^{n-1}\phi\circ T^{j}$ satisfies
\begin{equation}\label{equ_hyperbolicity definition}
\sup_{X}\frac{1}{n}S_{n}(\phi)<P(T,\phi),~~\mbox{where $P(T,\phi)$
is the topological pressure.}
\end{equation}
The notation of hyperbolic potential 
has been extensively used in a number of works include the studies
of H\"{o}lder continuous potentials in both complex rational maps
\cite{Hay99,DPU96,DU91a,Prz90,InoRiv12,LiRiv13a} and real interval maps with
sufficient regularity \cite{LiRiv13a,LiRiv13b}. Let us mention that
these studies also include the classical result of Lasota and Yorke
\cite{LY73} where $f$ is assumed to be piecewise $C^{2}$ and
uniformly expanding, and $\phi=-\log|Df|$.

Comparing with the inducing schemes approaches (see the survey
\cite{IT13,SUZ11} and the references therein), demonstrations of the
hyperbolicity turn out to possess two advantages. On one hand, it is
shown in \cite[Theo A]{LiRiv13a} that every H\"{o}lder continuous
potential is hyperbolic provided that the interval map $f$ satisfies
the weak regularity assumptions, that is all periodic points of $f$
are hyperbolic repelling and for every critical value of $f$,
$$
\lim_{n\to\infty}|Df^{n}(v)|=+\infty.
$$
We also refer to \cite{InoRiv12} for a complex version of the above
property. On the other hand, hyperbolicity of the potentials
directly yields that the corresponding transfer operators is bounded
and quasi-compact in the real setting \cite{Kel85,LiRiv13b}, or is
quasi-periodic in the complex setting \cite{FLM83,Man83}
respectively. In both settings, this approach gives a good
understanding of equilibrium states and their statistical
properties.



Given a differentiable one-dimensional dynamical system $f:X\to X$,
hyperbolicity of the potentials also plays an important role in
studies on geometric potentials (i.e., $\phi:=-t\log|Df|,~t\in\R$).
Note that every upper semi-continuous potential $\phi:X\to
\R\cup\{-\infty\}$ has the following agreeable property (shown in
Lemma \ref{lem_uppersemi}): $\phi$ is hyperbolic if and only if
\begin{equation}\label{equ_hyperbolicityequilivalent}
\sup_{\nu\in\cM(T,X)}\int\phi d\nu<P(T,\phi),
\end{equation}
where $\cM(T,X)$ is the set of Borel $T$-invariant probability
measures.

Put
$$
\chi_{\inf}:=\inf_{\nu\in\cM(f,X)}\int\log|Df|d\nu,~~~~~~~~~
\chi_{\sup}:=\sup_{\nu\in\cM(f,X)}\int\log|Df|d\nu,
$$
and define
$$
t_{-}:=\inf\big\{t\in\R:P(f,-t\log|Df|)+t\chi_{\sup}(f)>0\big\};
$$
$$
t_{+}:=\sup\big\{t\in\R:P(f,-t\log|Df|)+t\chi_{\inf}(f)>0\big\}.
$$
In the view of \eqref{equ_hyperbolicityequilivalent}, potential
$-t\log|Df|$ is hyperbolic when $t\in(t_{-},t_{+})$. Using this
fact, Przytycki and Rivera-Letelier showed the real analyticity of
the pressure function $P(f,-t\log|Df|)$ at $(t_{-},t_{+})$ for
complex rational maps \cite{PRL11} and real interval maps
\cite{PRL13} respectively. We also highlight the work of Makarov and
Smirnov \cite{MakSmi00}. They established an elegant expression of
the pressure function at $t<0$ in terms of hidden pressure and
Lyapunov exponent of periodic points for arbitrary rational maps
\cite[Theo A]{MakSmi00}, and provided an explicit characterization
\cite[Theo B]{MakSmi00} at which the number $t_{-}$ is finite. Their
methods involve a refine study on the hyperbolicity of the
potential, and are closed related to a former paper by Ruelle
\cite{Rue92} (see also an unpublished note from Smirnov
\cite{Smi99}). To simplify the notation, we use a convention
\emph{Makarov-Smirnov's formalism} to stand for the statements of
Theorem A and Theorem B in \cite{MakSmi00} henceforth.


In this paper, we study further about the thermodynamic formalism
for a subclass of upper semi-continuous potentials. In particular,
we will generalize the studies on the hyperbolicity of H\"{o}lder
continuous potentials, and obtain an analogous Makarov-Smirnov's
formalism, but in the setting of interval maps with sufficient
regularity. We stress that the proof in \cite[Theo A, B]{MakSmi00}
using the Sobolev spaces to create spectral gap of the corresponding
transfer operator can not be directly applied to our interval
setting, as it heavily relies on the fact that a non-constant
complex rational map is \emph{open}, i.e., the open sets has open
images. However, interval maps are not open in general, so the
corresponding transfer operator is usually not invariant even under
the space of continuous functions (see \cite[Ex2.2]{RL15}). Instead,
our methodology will be closely related to the recent results
developed in \cite{InoRiv12,LiRiv13a,LiRiv13b}. The difference on
the proofs between real and complex setting will be discussed in
\S\ref{sec_further researches}.



Let us proceed with the exact definitions and statements.

\bigskip

\subsection{Precise statement}
We begin by briefly introducing the main objects. Let $I$ be a
compact interval in $\R$. For a differentiable map $f:I\to I$, a
point of $I$ is \emph{critical} if the derivative of $f$ vanishes at
it. We denote by $\mbox{Crit}(f)$ the set of critical points of $f$.
We also denote by $J(f)$ the \emph{Julia set}, which is the set of
$x\in I$ with the following property: for every neighborhood $V$ of
$x$, the family $\{f^{n}|_{V}\}_{n=0}^{\infty}$ is not
equi-continuous. Let $\mbox{Crit}'(f):=\mbox{Crit}(f)\cap J(f)$. Let
also $\mbox{Per}(f)$ be the set of periodic points.

In what follows, we denote by $\mathscr{A}$ the collection of all
non-injective differentiable maps $f:I\to I$ such that
\begin{itemize}
  \item The critical set is finite;
  \item $Df$ is H\"{o}lder continuous;
  \item Each critical point $c\in\mbox{Crit}(f)$ is \emph{non-flat}, i.e., there exist a number
  $\ell_{c}\geq1$ and diffeomorphisms $\varphi$ and $\psi$ with $D\varphi,D\psi$ H\"{o}lder continuous,
  such that $\varphi(c)=\psi(f(c))=0$, and such that on a neighborhood
  of $c$ on I, we have
  $$
|\psi\circ f|=\pm|\varphi|^{\ell_c};
  $$
  (Such $\ell_{c}$ is usually called the \emph{local order} of the critical
  point).
  \item The Julia set $J(f)$ is \emph{completely invariant} \footnote{In contrast with complex rational maps, the Julia set of an interval map might not completely
  invariant.
  However, it is possible to make an arbitrarily small smooth perturbation of $f$ outside a neighborhood, so that the Julia set of the perturbed map is completely invariant, and coincides with $J(f)$ correspondingly.
  We refer to \cite{dMvS93} for more detailed background on the theory of Julia set in the real setting.}, i.e., $f(J)=f^{-1}(J)=J$, and contains
  at least two points;
  \item Every point $p\in\mbox{Per}(f)$ is \emph{hyperbolic repelling}, i.e., every periodic point $p$ of periodicity $N$
  has
$|D(f^{N})(p)|>1$;
  \item $f$ is \emph{topologically exact} on the Julia set $J(f)$, i.e., for each open set $U\in J(f)$, there exists an integer $n\geq0$ such that $J(f)\subset f^{n}(U)$.
\end{itemize}
It is clear that the set $\sA$ contains the family of smooth
non-degenerated interval maps which are topologically exact on the
$J(f)$ and have no neutral cycles. Note that every map $f\in\sA$,
the Juila set $J(f)$ has no isolated point, is the complement of the
basin of periodic attractors, and contains all interesting part of
the dynamics.

Throughout the rest of the paper, for each $f\in\sA$, we restrict
the action of $f$ to its Julia set $f|_{J(f)}:J(f)\to J(f)$.

\medskip

On the other hand, let us recall some basic setting of thermodynamic
formalism, referring the interested reader to \cite{Kel98} or
\cite{PU11} for more information.

Let $(X,d)$ be a compact metric space, and $T:X\to X$ be a
continuous map with topological entropy $h_{top}(T)<\infty$. Denote
by $\cM(X)$ the space of Borel probability measures on $X$ endowed
with the weak* topology, and let $\cM(T,X)$ denote the subset of
$T-$invariant ones. For each measure $\nu\in\cM(T,X)$, denote by
$h_{\nu}(T)$ the \emph{measure-theoretic} entropy of $\nu.$

Given a upper semi-continuous\footnote{A function
$\phi:X\to\R\cup\{-\infty\}$ is \emph{upper semi-continuous} if the
sets $\{y\in X: \phi(y)<c\}$ are open for each $c\in\R$. Since $X$
is compact, $\sup\phi<+\infty$.} function
$\phi:X\to\R\cup\{-\infty\}$, the \emph{free energy} of $T$ for the
\emph{potential} $\phi$ is defined as
$$
\cF_{\nu}(T,\phi):=h_{\nu}(T)+\int_{X}\phi d\nu.
$$
The \emph{pressure function} $P(T,\phi)$ can be defined by means of
the \emph{variational principle:}
$$
P(T,\phi):=\sup_{\nu\in\cM(T,X)}\cF_{\nu}(T,\phi).
$$
An \emph{equilibrium state} of $T$ for the potential $\phi$ is a
measure $\nu\in\cM(T,X)$ which satisfies
$\cF_{\nu}(T,\phi)=P(T,\phi)$. If the function $\mu\to h_{\mu}$
upper semi-continuous under the weak* topology, then such
equilibrium states exist.

The \emph{hidden pressure function} $\widetilde{P}(T,\phi)$ is
obtained by restricting the admissible measures to the set
$\widetilde{\cM}(T,X)$ of all invariant \emph{non-atomic} measures:
$$
\widetilde{P}(T,\phi):=\sup_{\nu\in\widetilde{\cM}(T,X)}\cF_{\nu}(T,\phi).
$$
Since $\widetilde{\cM}(T,X)$ is not compact under weak* topology,
the hidden pressure may or may not attain its supremum. If the
former case occurs, we say a non-atomic measure
$\nu\in\widetilde{M}(T,X)$ is a \emph{hidden equilibrium state} if
$\cF_{\nu}(T,\phi)=\widetilde{P}(T,\phi)$.

\medskip

Given an interval map $f:J(f)\to J(f)$ in $\sA$, denote by
$\mbox{USC}(J(f))$ the set of all upper semi-continuous functions
from $J(f)$ to $\R\cup\{-\infty\}$. In this paper, we are
particularly interested in a subclass $\cU\subset\mbox{USC}(J(f))$
\begin{equation}\label{equ_uppersemicontinuous1}
    \begin{split}
\cU:=&\Big\{u\in\mbox{USC}(J(f)):u(x)=g(x)+\sum_{c\in\tiny{\mbox{Crit}'}(f)}b_{u}(c)\log|x-c|,\\
&~\mbox{with}~g~\mbox{H\"{o}lder continuous, and~} b_{u}(c)\geq
0\Big\}.
\end{split}
\end{equation}
It is clear that every H\"{o}lder or geometric potential belongs to
the set $\cU$. For each $u\in\cU$, and $\nu\in\cM(f,J(f))$, put
$\theta_{\nu}:=\int_{J(f)}u d\nu$. In particular for each periodic
point $a\in\mbox{Per}(f)$ of periodicity $n$, put
$\theta_{a}:=\frac{1}{n}\sum_{j=0}^{n-1}u\circ f^{j}(a)$, and
$$
\theta_{\max}:=\sup\{\theta_{a},~~a\in\mbox{Per}(f)\}.
$$
Consider $\theta_{\nu},\theta_{a},\theta_{\max}$ with $u$ replaced
by $\log|Df|$. These respectively give
$\chi_{\nu},\chi_{a},\chi_{\max}$, which are called \emph{Lyapunov
exponent (point-wise,maximum)}.

%

\medskip

Our first main result concerning the behavior of the pressure
function $P(f,-tu)$ for $t<0$ is the following.

\begin{theorem}\label{theo_pressurefunction}
Let $f:J(f)\to J(f)$ be an interval map in $\sA$, and
$u:J(f)\to\R\cup\{-\infty\}$ be a upper semi-continuous potential in
$\cU$. If the following hypothesis satisfies,
\begin{itemize}
  \item[$(*)$] For each $t<0$, if $\widetilde{P}(f,-tu)$ attains the supremum then every hidden equilibrium state $\mu$ of $f$ for the
potential $-tu$ has strictly positive Lyapunov exponent.
\end{itemize}
then the hidden pressure function $\widetilde{P}(f,-tu)$ is real
analytic on $(-\infty,0)$, and
\begin{equation}\label{equ_Mararov-Smirnov formalism}
P(f,-tu)=\max\{\widetilde{P}(f,-tu),-t\theta_{\max}\},~~\forall t<0.
\end{equation}
\end{theorem}

\medskip

Analogous to the complex rational maps, the pressure function
$P(f,-tu)$ may or may not be analytic, and the first possibility is
easier to construct. If $P(f,-tu)$ is not real analytic on
$(-\infty,0)$, We say that $f$ has a \emph{phase transition}. Our
second main result gives an explicit characterization on the
appearance of phase transitions. To state this, we will need the
following definition. For each upper semi-continuous function
$u\in\cU$, put
\begin{equation}\label{equ_singluar set}
\Lambda(u):=\{c\in\mbox{Crit}'(f):~b_{u}(c)>0\},
\end{equation}
and $u$ is said to be \emph{exceptional} for $f$, if there is a
non-empty forward invariant finite subset $\Sigma\subset J(f)$,
satisfying
\begin{equation}\label{equ_invariance}
\emptyset\neq f^{-1}(\Sigma)\backslash\Sigma\subset \Lambda(u).
\end{equation}
Such finite set $\Sigma$ is called a \emph{$\Lambda(u)$-exceptional
set}, or simply an exceptional set. If no such exceptional set
exists, then the potential $u$ is said to be \emph{non-exceptional}.
Analogous to that in the complex setting, denote by
$\Sigma_{\max}^{(u)}$ the union of all $\Lambda(u)$-exceptional
sets. We highlight that $\Sigma_{\max}^{(u)}$ will be shown as a
$\Lambda(u)$-exceptional set and contains at least one periodic
orbit in Corollary \ref{coro_universal bound}. Therefore, we can
define
$$
\theta_{*}:=\max\{\theta_{a}:a\in\Sigma_{\max}^{(u)}\cap\mbox{Per}(f)\}.
$$
In fact, we will provide a universal bound on the cardinality of
$\Sigma_{\max}^{(u)}$ (see Corollary \ref{coro_universal bound} and
the remarks \ref{rem_lowerbound} and \ref{rem_parallel} afterwards)
in terms of the cardinality of $\Lambda(u)$ for every map $f$ in
$\sA$. This is a more refined estimation than that in \cite{PRL13},
and can be viewed as a parallel property to that in the complex
setting: every exceptional set of a rational function has at most 4
elements \cite{MakSmi96,MakSmi00,GPRRL13}. We believe this refined
estimation will have its independent interest.

\medskip

The second main result is as follows.
\begin{theorem}\label{theo_phasetransition}
Let $f:J(f)\to J(f)$ be an interval map in $\sA$, and
$u:J(f)\to\R\cup\{-\infty\}$ be a upper semi-continuous potential in
$\cU$. If hypothesis $(*)$ satisfies, then the pressure function
$P(f,-tu)$ has a phase transition on $(-\infty,0)$ if and only if
$u$ is exceptional and
$$
\theta_{*}>\sup\{\theta_{\nu}:\nu\in\cM(f,J(f)),~~\nu(\Sigma_{\max}^{(u)})=0\}.
$$
\end{theorem}

\medskip

\subsection{Reductions} Before stating our strategy of the proof,
we discuss a few interesting corollaries derived from Theorem
\ref{theo_pressurefunction} and Theorem \ref{theo_phasetransition}.

On one hand, restricting the potential $u$ to be a H\"{o}lder
continuous potential, we have
\begin{coro}\label{coro_holder}
Let $f:J(f)\to J(f)$ be an interval map in $\sA$, and $-tu:J(f)\to
\R$ be a hyperbolic H\"{o}lder continuous potential for every $t<0$,
then $P(f,-tu)$ has no phase transitions and equals to
$\widetilde{P}(f,-tu)$ on $(-\infty,+\infty)$.
\end{coro}
This corollary is actually the main theorem shown in \cite[Theo
B]{LiRiv13b} when restricting to maps with non-flat critical points, and we state below the proof briefly. Firstly, note
that by our choice of $f\in\sA$, it follows from \cite[Lem2.3]{BK98} that the function $\mu\to h_{\mu}(f)$ is
upper semi-continuous with respect to $weak^*$ topology, so the equilibrium states of $f$ for
$-tu$ with $t<0$ always exists. Note also every measure supported on
a periodic orbit has positive Lyapunov exponent, thus hypothesis
$(*)$ implies the assertion that every equilibrium state has
positive Lyapunov exponent. Together with \cite[Theo
1]{Li14}, the latter assertion is equivalent to the assumption that
the potential $-tu$ is hyperbolic. Conversely, the
hyperbolicity of $-tu$ implies hypothesis $(*)$. Therefore,
hypothesis $(*)$ is equivalent to the assumption that $-tu$ is
hyperbolic, for every $t<0$.

Secondly, for every H\"{o}lder continuous potential $u$, we have
$\Lambda(u)=\emptyset$, and thus $u$ is non-exceptional. Therefore,
Theorem \ref{theo_pressurefunction} and Theorem
\ref{theo_phasetransition} yield Corollary \ref{coro_holder} for
$t\leq0$. As readers will notice later, since $-tu$ is also
H\"{o}lder continuous for the positive values of $t$, the same proof
also works for the remaining part. Hence we obtain Corollary
\ref{coro_holder} completely.
\begin{rema}
As shown in \cite{LiRiv13b}, we highlight that the non-flatness
hypothesis of critical points in the definition of class $\sA$ is
not needed for the validity of Corollary \ref{coro_holder}, but it
will be fundamental for our proof of the general upper
semi-continuous potential in $\cU$.
\end{rema}

\medskip

On the other hand, restricting the potential $u$ to be a geometric
potential $-t\log|Df|$, we will show the hypothesis $(*)$ is
automatically satisfied for every $t<0$ (see Lemma \ref{lem_positive
on the Lyap}). Thus Theorem \ref{theo_pressurefunction} and Theorem
\ref{theo_phasetransition} are reduced to the Makarov-Smirnov's
formalism for interval maps. Form now on, put $P(t):=
P(f,-t\log|Df|)$ and $\widetilde{P}(t):=
\widetilde{P}(f,-t\log|Df|)$ for every $t<0$. Moreover, put
$\Sigma_{\max}:=\Sigma_{\max}^{(\log|Df|)}$ and $
\chi_{*}:=\max\{\chi_{a}:a\in\Sigma_{\max}\cap\mbox{Per}(f)\}.$

\medskip

\begin{coro}[A real version of Makarov-Smirnov's formalism]\label{coro_geometric}
Let $f:J(f)\to J(f)$ be an interval map in $\sA$, then
\begin{enumerate}
  \item the hidden pressure $\widetilde{P}(t)$ is real analytic on
  $(-\infty,0)$, and $$
  P(t)=\max\{\widetilde{P}(t),-t\chi_{\max}\},~~\forall
  t<0;$$
  \item $P(t)$ has a phase transition on $(-\infty,0)$, if and only
  if $\log|Df|$ is exceptional and satisfies
  $$
  \chi_{*}>\sup\{\chi_{\nu}:\nu\in\cM(f,J(f)),~~\nu(\Sigma_{\max})=0\}.
  $$
  In particular, if $\log|Df|$ is non-exceptional, then
  $P(t)$ equals to $\widetilde{P}(t)$ for every $t<0$, and has no phase transition on
  $(-\infty,0)$\footnote{The last assertion in Statement $(b)$ is also proved by \cite[Theo B]{IT10} and \cite[\S2.2]{GPR14},
  but it seems that our hypothesis is weaker in the sense that it does not rely on any bounded distortion
  hypothesis, and our approach is different from
  theirs.}.
\end{enumerate}
\end{coro}
In what follows, we consider only the proof of Corollary
\ref{coro_geometric} directly. One of our novelties in our proof is
we show that for the potential $-t\log|Df|$ with $t<0$, the
non-exceptionality implies the its hyperbolicity. This fact allows
some of the arguments or statements shorter and simpler. As readers
will notice, there is no difficulty in extending our proof of
Corollary \ref{coro_geometric} to general upper semi-continuous
potential in $\cU$ and obtaining Theorem \ref{theo_pressurefunction}
and Theorem \ref{theo_phasetransition}\footnote{As a contrast, it is
however unknown whether there is an analogous complex version of
Theorem \ref{theo_pressurefunction} and Theorem
\ref{theo_phasetransition}, although the original Makarov-Smirnov's
formalism in the complex setting is obtained. The reasons will be
explained in \S\ref{sec_further researches}.}.


\medskip

\subsection{Proof strategy} Let us describe briefly about our
proof strategy of Corollary \ref{coro_geometric}. The proof actually
unifies and adapts several machineries used in
\cite{MakSmi00,InoRiv12,LiRiv13a,LiRiv13b}. To be more precise, we
first take a co-homology transformation on the geometric potential
and identify the so-called ``Key Lemma'' as follows.
\begin{theorem}[Key Lemma]
Let $f:J(f)\to J(f)$ be an interval map in $\sA$,
then there exists a lower semi-continuous function
$h:J(f)\to\R\cup\{+\infty\}$ such that
\begin{enumerate}
  \item[$(1)$] $h$ only has log poles in $\Sigma_{\max}$;
  \item[$(2)$] Let $G:=\log|Df|+h\circ f-h$, then $G\in\cU$ and $G$ is non-exceptional.
\end{enumerate}
Moreover,
\begin{enumerate}
  \item
  $\widetilde{P}(f,-tG)=\widetilde{P}(t),~~\forall t<0;$
  \item  For each ergodic invariant probability measure $\mu$ with strictly positive Lyapunov exponent, there
  is a full $\mu-$measure subset $X\subset J(f)$ such that
  \begin{equation}\label{equ_hyperbolicity}
  \limsup_{n\to\infty}\frac{1}{n}\log\sum_{y\in
  f^{-n}(x_{0})}\exp(S_{n}(-tG)(y))>\int_{J(f)}-tGd\mu,~~\forall
  x_{0}\in X,~~\forall t<0.
  \end{equation}
\end{enumerate}
\end{theorem}
\begin{rema}
We remark that if further assuming $\log|Df|$ is non-exceptional,
then $G=\log|Df|$ and $h=0$, and the non-flatness hypothesis of the
critical points in the definition of class $\sA$ is not
required for the validity of Key Lemma. 

\end{rema}

Back to the strategy description, on one hand, Statement $(a)$
transfers the studies on the hidden pressure information from the
potential $\log|Df|$ to the new upper semi-continuous potential $G$.
On the other side, Statement $(b)$ is used to deduce the
hyperbolicity of $G$. After obtaining the hyperbolicity, we can use
the Patterson-Sullivan method to construct a non-atomic conformal
measure, and can apply a certain Keller's space to create a spectral
gap for the corresponding transfer operator induced by the new
potential $G$. These properties will be the main ingredients on
deducing the assertions of Corollary \ref{coro_geometric}.

We also remark that some regularity restrictions about the
potentials are required on applying the Patterson-Sullivan
construction and Keller's result. In fact, this is the part in the
proof where the potential $u$ is required to be in the subclass
$\cU$. Currently, we are unable to obtain Theorem
\ref{theo_pressurefunction} and \ref{theo_phasetransition} for
arbitrary upper semi-continuous potentials.

\medskip

\subsection{Main ideas of the proofs and organization of the paper}
The rest of the paper is organized as follows. In section
\ref{sec_normality}, we discuss a few properties about the normality
and their relations, which are basic to the arguments that follow.
Section \ref{sec_exceptional}, \ref{sec_IMFS} and
\ref{sec_keylemmaperiod} are devoted to the proof of Key Lemma. In
Section \ref{sec_exceptional}, we study in detail about the
co-homology transformation. Based on this, we provide an explicit
construction on the new potential $G$, and prove Statement $(a)$. In
Section~\ref{sec_IMFS}, a refined iterated multi-valued function
system is constructed in order to prove Statement $(b)$ for the case
where $\mu$ is non-atomic. The rest case where $\mu$ is atomic is
dealt in Section \ref{sec_keylemmaperiod}. We use Key Lemma in
Section \ref{sec_hyperbolic and conformal} to show the hyperbolicity
of the new potential $G$, and construct a non-atomic conformal
measure. These results are used in Section
\ref{sec_makarov-smirnovformalism} to construct certain Keller's
spaces, where the corresponding transfer operators acting on have
spectral gaps. Thus the hidden pressure function has no phase
transition. In the final section, we discuss a few relations with
the Makarov-Smirnov's formalism in the complex setting under the
view of our Key Lemma. In particular, we reprove \cite[Theo
3.1]{MakSmi00} solely via $BV_{2}$ functional spaces, and conjecture
that our result might be useful on reobtaining the original
Makarov-Smirnov's formalism by an alternative proof without
introducing Sobolev spaces, provide that we can obtain a complex
version of our Theorem \ref{theo_pressurefunction} and
\ref{theo_phasetransition}. The appendix is devoted to provide some
basic background on Keller's spaces which are introduced by
\cite{Kel85} (see also \cite{RL15}).
%

\medskip

\begin{ack}
Y.Z. would like to express his great thanks to Juan Rivera-Letelier
for his inspiring discussions. Y.Z.
also thanks Huaibin Li for some helpful comments in the early
versions of this manuscript.
\end{ack}

\section{Normality}\label{sec_normality}
The main goal of this section is to show Proposition
\ref{lem_postercritical} below, which will be of great importance in
the proof of Key Lemma substantially.

Given an interval map $f\in\sA$, and a subset
$\Lambda\subseteq\mbox{Crit}'(f)$, a point $x\in J(f)$ is
\emph{$\Lambda$-normal} or simply \emph{normal}, if for every integer
$n\geq1$, there is a pre-image $y$ of $x$ by $f^{n}$, such that
\begin{equation}\label{def_normality}
\{y,f(y),\cdots,f^{n-1}(y)\}\cap\Lambda=\emptyset.
\end{equation}
Otherwise, this point $x$ is said to be \emph{$\Lambda$-abnormal} or
simply \emph{abnormal}.

\medskip

\begin{prop}\label{lem_postercritical}
Let $f:J(f)\to J(f)$ be an interval map in $\sA$, and $\Lambda$ be a
subset of $\mbox{Crit}'(f)$. If every non-empty forward invariant
finite set $\Sigma$ satisfies
\begin{equation}\label{equ_nonexceptional}
f^{-1}(\Sigma)\backslash\Sigma\nsubseteq\Lambda,
\end{equation}
then for each point $x\in J(f)$, there is an integer $N\geq0$ such
that $f^{N}(x)\notin\Lambda$ and $f^{N}(x)$ is $\Lambda$-normal.

In addition, if further assuming $x$ is periodic, then $x$ itself is
neither in $\Lambda$ nor $\Lambda-$normal.
\end{prop}
The proof of Proposition \ref{lem_postercritical} depends on several
Lemmas, and will be given in the end of this section.

\medskip

For brevity, given a map $\pi:\Xi\to\hat{\Xi}$, denote by $
\maltese$ the \emph{rank 1 pre-image set},
$\maltese:=\{\xi\in\Xi:\sharp\pi^{-1}(\xi)=1\}$, where $\sharp A$
means the cardinality of a finite set $A$.

\begin{lemm}\label{lem_tec}
Let $\hat{\Xi}$ be a finite set, $\Xi$ a subset of $\hat{\Xi}$,
$\pi:\Xi\to\hat{\Xi}$ a map, and $\maltese$ is the resulting rank 1
pre-image set, then
\begin{equation}\label{equ_relationship on cardnality}
\sharp(\Xi\cap\pi(\Xi))\leq
3\sharp(\Xi\backslash\pi(\Xi))+\sharp(\maltese\cap\pi(\maltese)).
\end{equation}
\end{lemm}
\begin{proof}
We proceed by induction on the cardinality of $\Xi$. The case
$\sharp\Xi=1$ is trivial. Let $n\geq 2$ be an integer such that the
desired inequality is satisfied for every set of cardinality $n-1$.
Let $\Xi$ be a set of cardinality $n$, and let $\pi$ and $\maltese$
be as in the statement of the lemma.

For each $\xi$ in $\Xi$ we define an integer $n(\xi)$ as follows: if
$\xi$ is in $\hat{\Xi}\backslash\Xi$, or if $\xi$ is periodic for
$\pi$, put $n(\xi)=0$; if $\xi$ is in $\Xi$ and is not periodic for
$\pi$, let $n(\xi)$ be the least integer $n\geq 1$ such that
$\pi^n(\xi)$ is in $\hat{\Xi}\backslash\Xi$ or is periodic for
$\pi$. Moreover, put
$$
n(\Xi):=\max\{n(\xi):\xi\in\Xi\}.
$$

If $n(\Xi)=0$, then every point of $\Xi$ is periodic for $\pi$, so
$\pi(\Xi)=\Xi$, and therefore $\maltese=\Xi$. So the desired
inequality is verified in this case.

Suppose $n(\Xi)\geq 1$, let $\tilde{\xi_0}$ in $\Xi$ be such that
$n(\tilde{\xi_0})=n(\Xi)$, and put
$\tilde{\xi}:=\pi(\tilde{\xi_0})$. There are 3 cases.

{\bf Case 1.} $n(\Xi)=1$. Then $\tilde{\xi}$ is in
$\hat{\Xi}\backslash\Xi$, or is periodic for $\pi$. In both cases
the set
$$
\Delta:=\{\xi\in\pi^{-1}(\tilde{\xi}):n(\xi)=1\}
$$
is nonempty and disjoint from $\pi(\Xi)$. If $\Delta$ is equal to
$\Xi$, then $\pi(\Xi)\cap\Xi$ is empty and the desired inequality is
trival in this case. So we assume that the set
$\tilde{\Xi}:=\Xi\backslash\Delta$ is nonempty. Put
$$
\tilde{\pi}:=\pi|_{\tilde{\Xi}} \mbox{~and~}
\tilde{\maltese}:=\{\xi\in\tilde{\Xi}:\sharp\tilde{\pi}^{-1}(\xi)=1\}.
$$
Clearly, we have
$$
\tilde{\Xi}\cap\tilde{\pi}(\tilde{\Xi})=\Xi\cap\pi(\Xi).
$$
On the other hand, if $\tilde{\xi}$ is in $\hat{\Xi}\backslash\Xi$,
then $\tilde{\maltese}=\maltese$, and if $\tilde{\xi}$ is periodic
for $\pi$, then $\tilde{\maltese}=\maltese\cup\{\tilde{\xi}\}$, and
therefore
$$
\tilde{\maltese}\cap\tilde{\pi}(\tilde{\maltese})=\left(\maltese\cup\{\tilde{\xi}\}\right)\cap\left(\pi(\maltese)\cup\{\pi(\tilde{\xi})\}\right)\subset(\maltese\cap\pi(\maltese))\cup\{\tilde{\xi},\pi(\tilde{\xi})\}
$$
In both cases we have
$$
\sharp(\tilde{\maltese}\cap\tilde{\pi}(\tilde{\maltese}))\leq\sharp(\maltese\cap\pi(\maltese))+2.
$$
So, by the induction hypothesis and the fact that $\Delta$ is
nonempty and disjoint from $\pi(\Xi)$, we have
\begin{equation*}
\begin{array}{ccl}
\sharp(\Xi\cap\pi(\Xi))=\sharp(\tilde{\Xi}\cap\tilde{\pi}(\tilde{\Xi}))& \leq & 3\sharp(\tilde{\Xi}\backslash\tilde{\pi}(\tilde{\Xi}))+\sharp(\tilde{\maltese}\cap\tilde{\pi}(\tilde{\maltese})\\
& \leq & 3(\sharp(\Xi\backslash\pi(\xi))-\sharp\Delta)+\sharp(\maltese\cap\pi(\maltese))+2\\
& \leq &
3\sharp(\Xi\backslash\pi(\Xi))+\sharp(\maltese\cap\pi(\maltese))-\sharp\Delta.
\end{array}
\end{equation*}
This implies the desired inequality for $\pi$, and completes the
proof of the induction step in the case $n(\Xi)=1$.

\medskip

{\bf Case 2.} $n(\Xi)\geq 2$ and $\tilde{\xi}\notin\maltese$. In
this case the set $\pi^{-1}(\tilde{\xi})$ is disjoint from
$\pi(\Xi)$ and the set
$\tilde{\Xi}:=\Xi\backslash\pi^{-1}(\tilde{\xi})$ is nonempty. On
the other hand, if we put
$$
\tilde{\pi}:=\pi|_{\tilde{\Xi}} \mbox{~and~}
\tilde{\maltese}:=\{\xi\in\tilde{\Xi}:\sharp\tilde{\pi}^{-1}(\xi)=1\},
$$
then
$$
\tilde{\pi}(\tilde{\Xi})=\pi(\Xi)\backslash\{\tilde{\xi}\}
\mbox{~and~} \tilde{\maltese}=\maltese.
$$
Therefore we have
$$
\tilde{\Xi}\cap\tilde{\pi}(\tilde{\Xi})=(\Xi\cap\pi(\Xi))\backslash\{\tilde{\xi}\},
$$
$$
\tilde{\Xi}\backslash\tilde{\pi}(\tilde{\Xi})=\left(\Xi\backslash\pi^{-1}(\tilde{\xi})\right)\backslash\left(\pi(\Xi)\backslash\{\tilde{\xi}\}\right)=\left((\Xi\backslash\pi(\Xi))\cup\{\tilde{\xi}\}\right)\backslash\pi^{-1}(\tilde{\xi}),
$$
and therefore
$$
\sharp\left(\tilde{\Xi}\backslash\tilde{\pi}(\tilde{\Xi})\right)=\sharp(\Xi\backslash\pi(\Xi))-(\sharp\pi^{-1}(\tilde{\xi})-1).
$$
So, by the induction hypothesis we have
\begin{equation*}
\begin{array}{ccl}
\sharp(\Xi\cap\pi(\Xi))-1& =& \sharp(\tilde{\Xi}\cap\tilde{\pi}(\tilde{\Xi})\\
& \leq & 3\sharp\left(\tilde{\xi}\backslash\tilde{\pi}(\tilde{\Xi})\right)+\sharp\left(\tilde{\maltese}\cap\tilde{\pi}(\tilde{\maltese})\right)\\
& = & 3\sharp\left(\Xi\backslash\pi(\Xi)\right)-3\left(\sharp\pi^{-1}(\tilde{\xi})-1\right)+\sharp\left(\maltese\cap\pi(\maltese)\right)\\
& \leq &
3\sharp(\Xi\backslash\pi(\Xi))+\sharp(\maltese\cap\pi(\maltese))-3
\end{array}
\end{equation*}
This implies the desired inequality for $\pi$, and completes the
proof of the induction step in the case $n(\Xi)\geq 2$ and
$\tilde{\xi}$ is not in $\maltese$.

\medskip

{\bf Case 3.} $n(\xi)\geq 2$ and $\tilde{\xi}\in\maltese$. In this
case $\tilde{\xi_0}$, defined above, is the unique element of
$\pi^{-1}(\tilde{\xi})$. Note that
$\pi^{n(\tilde{\xi_0})}(\tilde{\xi_0})$ is in
$\hat{\Xi}\backslash\Xi$ or is periodic for $\pi$. In particular, it
is not in $\maltese$. It follows that there is at least an integer
$n\geq 2$ such that $\pi^{n}(\tilde{\xi_0})$ is not in $\maltese$.
Put
\begin{eqnarray*}
\Delta:=\left\{\pi^{j}(\tilde{\xi_0}): j\in\{0,1,\cdots,n-1\}\right\},\\
\tilde{\Xi}:=\Xi\backslash\Delta, \tilde{\pi}:=\pi|_{\Xi},
\mbox{~and~}
\tilde{\maltese}:=\{\xi\in\tilde{\Xi}:\sharp\tilde{\pi}^{-1}(\xi)=1\},
\end{eqnarray*}
and note that
$$
\tilde{\Xi}\cap\tilde{\pi}(\tilde{\Xi})=(\Xi\cap\pi(\Xi))\backslash\left(\Delta\backslash\{\tilde{\xi_0}\}\right),
~~\tilde{\Xi}\backslash\tilde{\pi}(\tilde{\Xi})=(\Xi\backslash\pi(\Xi))\backslash\{\tilde{\xi_0}\},
$$
and
$$
(\maltese\cap\pi(\maltese))\cap\Delta=\left\{\pi^j(\tilde{\xi_0}):j\in\{2,\cdots,n-1\}\right\}.
$$

Suppose $n=n(\tilde{\xi_0})$. If $\tilde{\Xi}$ is empty, then we
have
$$
\Xi\cap\pi(\Xi)=\maltese\cap\pi(\maltese)=\Delta\backslash\{\tilde{\xi_0}\},
$$
and the desired inequality holds in this case. If $\tilde{\Xi}$ is
nonempty, then we have
$$
\tilde{\maltese}=\maltese\backslash\left(\Delta\backslash\{\tilde{\xi_0}\}\right),
\mbox{~and~}
\tilde{\maltese}\cap\tilde{\pi}(\tilde{\maltese})=(\maltese\cap\pi(\maltese))\backslash\left(\Delta\backslash\{\tilde{\xi_0}\}\right).
$$
So, by the induction hypothesis we have
\begin{equation*}
\begin{array}{ccl}
\sharp(\Xi\cap\pi(\Xi))-(n-1) & = & \sharp(\tilde{\Xi}\cap\tilde{\pi}(\tilde{\Xi})\\
& \leq & 3\sharp(\tilde{\Xi}\backslash\tilde{\pi}(\tilde{\Xi}))+\sharp(\tilde{\maltese}\cap\tilde{\pi}(\tilde{\maltese})\\
& = & 3\sharp(\xi\backslash\pi(\Xi))-3(n-1)+\sharp(\maltese\cap\pi(\maltese))-(n-1)\\
& =&
3\sharp(\Xi\backslash\pi(\Xi))+\sharp(\maltese\cap\pi(\Xi))-4(n-1).
\end{array}
\end{equation*}
This implies the desired inequality when $n=n(\tilde{\xi_0})$.

It remains to consider the case $n\leq n(\tilde{\xi_0})-1$. In this
case $\tilde{\Xi}$ is nonempty, and we have
$$
\tilde{\maltese}\subset\left(\maltese\cap\tilde{\Xi}\right)\cup\left\{\pi^n(\tilde{\xi_0})\right\},
$$
and
\begin{eqnarray*}
\tilde{\maltese}\cap\tilde{\pi}(\tilde{\maltese}) & \subset & \left(\maltese\cap\pi(\maltese)\cap\pi(\tilde{\Xi})\right)\cup\left\{\pi^n(\tilde{\xi_0}),\pi^{n+1}(\tilde{\xi_0})\right\}\\
& = &
\left(\maltese\cap\pi(\maltese)\backslash\left\{\pi^j(\tilde{\xi_0}):j\in\{2,\cdots,n-1\}\right\}\right)\cup\left\{\pi^n(\tilde{\xi_0}),\pi^{n+1}(\tilde{\xi_0})\right\},
\end{eqnarray*}
and therefore
$$
\sharp(\tilde{\maltese}\cap\tilde{\pi}(\tilde{\maltese}))\leq\sharp(\maltese\cap\pi(\maltese))-n+4.
$$
So, by the induction hypothesis we have
\begin{eqnarray*}
\sharp(\Xi\cap\pi(\xi))-(n-1) & =& \sharp(\tilde{\Xi}\cap\tilde{\pi}(\tilde{\Xi}))\\
& \leq & 3\sharp\left(\tilde{\Xi}\backslash\tilde{\pi}(\tilde{\Xi})\right)+\sharp\left(\tilde{\maltese}\cap\tilde{\pi}(\tilde{\maltese})\right)\\
& \leq & 3\sharp(\Xi\backslash\pi(\Xi))-3+\sharp(\maltese\cap\pi(\Xi))-n+4\\
& \leq &
3\sharp(\Xi\backslash\pi(\Xi))+\sharp(\maltese\cap\pi(\maltese))-(n-1).
\end{eqnarray*}
This implies the desired inequality for $\pi$, and completes the
proof of the induction step in the case $n(\Xi)\geq 2$ and
$\tilde{\xi}$ is in $\maltese$. The proof of the lemma is thus
complete.

\end{proof}

\medskip

\begin{lemm}\label{lem_rank1}
Let $f:J(f)\to J(f)$ be a continuous interval map with $\maltese$
the resulting rank 1 pre-image set. Put
$$
a_{\max}:=\max_{x\in J(f)}\{f(x)\}~~~\mbox{and}~~~
a_{\min}:=\min_{x\in J(f)}\{f(x)\}.
$$
If $f$ is topological exact on $J(f)$, then we have
\begin{equation}\label{equ_emptysetintersectionrank1}
f(\maltese)\cap\maltese\subseteq\{a_{\max},a_{\min},f(a_{\max}),f(a_{\min})\}.
\end{equation}
\end{lemm}
\begin{proof} We prove the lemma
by contradiction. Suppose on the contrast that there is a point
$x\in
\big(f(\maltese)\cap\maltese\big)\backslash\{a_{\max},a_{\min},f(a_{\max}),f(a_{\min})\}$.
By the definition of $\maltese$, there exist a unique pre-image
$y\in \maltese\backslash\{a_{\max},a_{\min}\}$ with $f(y)=x,$ and a
unique pre-image $z$ with $f(z)=y$. We first show that $x,y,z$ are
pairwise distinct. Otherwise, if say $x=y$ then $x$ is a fixed point
and its unique pre-image is itself. Hence
$\cup_{i=1}^{\infty}f^{-i}(x)=\{x\}$, which is a contradiction to
the topological exactness. The remaining cases ``$y\neq z$'' and
``$x\neq z$'' are similarly supplied.

Next, we show that
\begin{enumerate}
  \item $x$ must be an extreme value of both $f|_{[0,y]}$ and
  $f|_{[y,1]}$;
  \item $x$ is a maximum (resp. minimum) of $f|_{[0,y]}$, if and only if
$x$ is a minimum (resp. maximum) of $f|_{[y,1]}$.
\end{enumerate}
For Statement $(a)$, we only prove $f$ is an extreme value of
$f|_{[0,y]}$, the other case is similarly supplied. Suppose on the
contrast that $x$ is neither a maximum nor a minimum at
$f|_{[0,y]}$. Using intermediate value theorem and the complete
invariance of $J(f)$, there is another point $d$ in $J$ differs from
$y$ such that $f(d)=f(y)=x$, which is a contradiction to $x$ being
rank 1-pre-image. For Statement $(b)$, notice that
$x\notin\{a_{\max},a_{\min}\}$, so $x$ is neither a maximum nor a
minimum of both $f|_{[0,y]}$ and $f|_{[y,1]}$ simultaneously.
Analogously, we have
\begin{enumerate}
  \item[$(a')$] $y$ must be an extreme value of both $f|_{[0,z]}$ and
  $f_{[z,1]}$;
  \item[$(b')$] $y$ a maximum (resp. minimum) of $f|_{[0,z]}$, if and only
  if
$y$ is a minimum (resp. maximum) of $f|_{[z,1]}$.
\end{enumerate}
In the rest of the proof, we will show that $f$ has a proper forward
invariant open subset in $J(f)$. This is a contradiction to the
topological exactness, and the desired assertion follows. In order
to do that, we need to distinguish the order of $x,y,z$ on the real
line. There are two possibilities, namely Case ``$z<y$'' and Case
``$y<z$''. Without loss of generality, it is sufficient to prove
Case ``$z<y$''. The proof of the remaining case can be simply
deduced by interchanging the positions of $z$ and $y$ in the proof
below.

Using the hypothesis ``$z<y$'', we have
$$
J=([a_{\min},z]\cap J)\bigcup([z,y]\cap J)\bigcup([y,a_{\max}]\cap
J).
$$
In the following, we distinguish two cases.

{\bf Case 1.} $x$ is a minimum of $f|_{[0,y]}$. Thus $x$ must be a
maximum of $f|_{[y,1]}$. We claim that $y$ must be a minimum of
$f|_{[0,z]}$ and a maximum of $f|_{[z,1]}$. This claim can be proved
by contradiction. Otherwise, suppose $y$ is a maximum of
$f|_{[0,z]}$ and a minimum of $f|_{[z,1]}$. Due to our hypothesis
$z<y$, we have $[0,z]\subset[0,y]$ and $[y,1]\subset[z,1]$. However,
the former statement yields $x<y$, while the latter statement yields
$x>y$, which is evidently impossible. Therefore, we obtain the
claim, and thus we have $x<y$. In other words $x$ must be outside
$[y,a_{\max}]\cap J$. There are further two possibilities.
\begin{itemize}
  \item[{\bf Subcase 1.}] $x\in[z,y]\cap J$. Then $x$ must be a minimum of $f|_{[z,y]}$ and $y$
must be a maximum of $f|_{[z,y]}$. Hence,
$$f([z,y]\cap~J)\subseteq f([z,y])\cap~J=[x,y]\cap~J\subseteq [z,y]\cap~J.$$
This implies that $[z,y]\cap J$ is a proper open subset of $J$
forward invariant under $f$;
  \item[{\bf Subcase 2.}] $x\in[a_{\min},z]\cap J$. Then $y$ must be a minimum
of $f|_{[0,z]}$ and $x$ must be a maximum of $f|_{[y,1]}$. Hence,
$$
f^{2}([a_{\min},z]\cap ~J)\subseteq f([y,a_{\max}]\cap J)\subseteq
[a_{\min},x]\cap~J\subseteq [a_{\min},z]\cap J.
$$
This also implies that $[a_{\min},z]\cap J$ is a proper open subset
of $J$ forward invariant under $f^{2}$.
\end{itemize}
In both subcases, $f$ has a proper open and forward invariant subset
of $J$.
\medskip

{\bf Case 2.} $x$ is a maximum of $f|_{[0,y]}$. Thus $x$ must be a
minimum of $f|_{[y,1]}$. Followed by similarly arguments as above,
$y$ must be a maximum of $f|_{[0,z]}$ and a minimum of $f|_{[z,1]}$.
Hence $y<x$. In other words, $x$ must be in $[y,a_{\max}]\cap J.$
Therefore,
$$
f([y,a_{\max}]\cap J)\subseteq [x,a_{\max}]\cap J\subseteq
[y,a_{\max}]\cap J.
$$
This implies that $[y,a_{\max}]\cap J$ is a proper open subset of
$J$ forward invariant under $f$.

\medskip

As a conclusion, we show that $f$ has a proper open forward
invariant subset of $J(f)$ in both cases, which is a contradiction
to the topological exactness on $J(f)$. This contradiction yields
the desired assertion \eqref{equ_emptysetintersectionrank1}, and
completes the proof of this lemma.

\end{proof}

\medskip

We are now ready to prove Proposition \ref{lem_postercritical}.

\medskip

\begin{proof}[Proof of Proposition \ref{lem_postercritical}]
The proof is split into 2 parts. In part 1, put
\begin{equation}\label{equ_weakexceptional}
S_{\Lambda}:=\big\{x\in J(f),~x~\mbox{is $\Lambda$-abnormal}\big\},
\end{equation}
and we show that
\begin{equation}\label{equ_cardnality}
\sharp S_{\Lambda}\leq 3\sharp\Lambda+4.
\end{equation}
In part 2, we use \eqref{equ_cardnality} to show the desired assertions of the Proposition \ref{lem_postercritical}.

\medskip

{\bf 1.} We first proceed by contradiction to show that
\begin{equation}\label{equ_weakexceptionality}
f^{-1}(S_{\Lambda})\backslash S_{\Lambda}\subset\Lambda.
\end{equation}
Otherwise, there is a point $x\in f^{-1}(S_{\Lambda})\backslash
S_{\Lambda}$, but $x\notin\Lambda$. This implies that $x\notin
S_{\Lambda}\cup\Lambda$, but $f(x)\in S_{\Lambda}$. However, this is
a contradiction to the abnormality of $x$. Thus we obtain
\eqref{equ_weakexceptionality}.

On the other hand, put $\Xi:=S_{\Lambda}\cup f^{-1}(S_{\Lambda})$.
So
\begin{equation}\label{equ_set}
S_{\Lambda}\subset \Xi\cap f(\Xi)~~\mbox{and}~~\Xi\backslash
f(\Xi)\subset f^{-1}(S_{\Lambda})\backslash S_{\Lambda}.
\end{equation}
Recall $\maltese$ to be the rank 1 pre-image set of $f$. Therefore,
\begin{alignat*}{2}
\sharp S_{\Lambda}&\leq \sharp(\Xi\cap
f(\Xi))~~~~~~~~~~~~~~~~~&\text{(Using \eqref{equ_set})}\\
&\leq3\sharp(\Xi\backslash f(\Xi))+\sharp(\maltese\cap
f(\maltese))~~~&\text{(Using Lemma \ref{lem_tec})}\\
&\leq 3\sharp(f^{-1}(S_{\Lambda})\backslash
S_{\Lambda})+4~~~~~&\text{(Using \eqref{equ_set} and Lemma
\ref{lem_rank1})}\\
&\leq 3\sharp\Lambda+4 ~~~~~~~~~&\text{(Using Inequality
\eqref{equ_weakexceptionality})},
\end{alignat*}
which completes the proof of \eqref{equ_cardnality}.

\medskip

{\bf 2.} In this part, we first show the following claim.
\begin{claim}\label{claim_nearly}
For every $x\in J(f)$, there is a non-negative integer $N$, such
that $f^{N}(x)$ is $\Lambda$-normal.
\end{claim}
We proceed the proof of Claim \ref{claim_nearly} by contradiction.
Suppose on the contrast that there is a point $x\in J(f)$, such that
the forward orbit $\cO_{x}:=\{f^{n}(x)\}_{n=0}^{\infty}\subset
S_{\Lambda}$. Following from \eqref{equ_cardnality} in Part 1, the
orbit $\cO_{x}$ must be pre-periodic and finite. Put
$$
A_{x}:=\{y\in S_{\Lambda}:~~\exists m\geq0,~\mbox{such
that}~f^{m}(y)\in \cO_{x}\}.
$$

In order to obtain Claim \ref{claim_nearly}, it is sufficient to
verify that $A_{x}$ is non-empty, forward invariant finite and
satisfies
\begin{equation}\label{equ_exceptional}
f^{-1}(A_{x})\backslash A_{x}\subseteq\Lambda.
\end{equation}
This is a contradiction to the hypotheses
\eqref{equ_nonexceptional}, and thus Claim \ref{claim_nearly}
follows. The non-empty, forward invariance and finiteness of $A_{x}$
are straightforward from its definition. To verify
\eqref{equ_exceptional}, suppose on the contrast that there is a
point $z\in f^{-1}(A_{x})\backslash A_{x}$, but $z\notin\Lambda$.
This means that $z\notin A_{x}\cup\Lambda$, but $f(z)\in A_{x}$.
Hence $z\notin S_{\Lambda},$ and thus $f(z)\notin S_{\Lambda}$,
which is evidently impossible. This contradiction completes the
proof of \ref{equ_exceptional}, and thus we obtain Claim
\ref{claim_nearly}.
\medskip

Next, we show that the integer $N\geq0$ stated in Claim
\ref{claim_nearly} can be further adapted such that
$f^{N}(x)\notin\Lambda.$ There are two cases.
\begin{enumerate}
  \item[{\bf Case 1:}] $x$ is pre-periodic. Then let integer $N\geq0$ be the smallest
integer such that $f^{N}(x)$ being periodic. Note that
$\Lambda\subset\mbox{Crit}'(f)$ and there is no periodic critical
point in $J(f)$, thus $f^{N}(x)\notin\Lambda.$ Moreover, we further
claim that $f^{N}(x)$ is $\Lambda$-normal. In fact, following from
Claim \ref{claim_nearly}, there is an integer $\hat{N}\geq0$ such
that $f^{\hat{N}}(x)$ is periodic and $\Lambda$-normal. Then
$f^{\hat{N}+1}(x)$ is also $\Lambda$-normal. By periodicity,
$f^{N}(x)$ must be $\Lambda$-normal.
  \item[{\bf Case 2:}] $x$ is non-preperiodic.
Then there is a subsequence $\{n_{i}\}$ such that each
$f^{n_{i}}(x)$ is $\Lambda$-normal, and $f^{n_{i}}(x)\neq
f^{n_{j}}(x),\forall i\neq j$. By the finiteness of $\Lambda,$ there
is an integer $N\geq0$, such that $f^{N}(x)$ is $\Lambda$-normal and
not in $\Lambda$.
\end{enumerate}
In both cases, we obtain the desired assertions. Thus the proof of
Proposition \ref{lem_postercritical} is completed.
\end{proof}

\bigskip

We close this section mentioning the corollary below that perhaps
have independent interests. For each interval map $f\in\sA$, Recall
$\Sigma_{\max}$ the \emph{maximum exceptional set},
\begin{equation}\label{equ_maximum exceptional set}
\Sigma_{\max}:=\bigcup\big\{\Sigma\subset
J(f):\Sigma~\mbox{is}~\mbox{Crit'(f)-exceptional}\big\}.
\end{equation}
\begin{coro}\label{coro_universal bound}
Let $\sA_{n}:=\{f\in\sA,~~\mbox{Crit}'(f)\leq n\}$, then
\begin{equation}\label{equ_universal bound}
    \sup_{f\in\sA_{n}}\sharp\Sigma_{\max}(f)\leq 3n+4.
\end{equation}
\end{coro}
\begin{proof}
Given an interval map $f$ in $\sA_{n}$, on one hand, it is
straightforward to see that each $\mbox{Crit}'(f)$-exceptional set
$\Sigma$ is contained in $S_{\tiny{\mbox{Crit}'(f)}}$. On the other
hand, for every two $\mbox{Crit}'(f)$-exceptional sets
$\Sigma_{1},\Sigma_{2}$, we have
\begin{align*}
f^{-1}(\Sigma_{1}\cup \Sigma_{2})\backslash(\Sigma_{1}\cup
\Sigma_{2})&=[f^{-1}(\Sigma_{1})\cup
f^{-1}(\Sigma_{2})]\backslash(\Sigma_{1}\cup \Sigma_{2})\\
&=(f^{-1}(\Sigma_{1})\backslash(\Sigma_{1}\cup \Sigma_{2}))\cup(f^{-1}(\Sigma_{2})\backslash(\Sigma_{1}\cup \Sigma_{2}))\\
&\subseteq(f^{-1}(\Sigma_{1})\backslash
\Sigma_{1})\cup(f^{-1}(\Sigma_{2})\backslash
\Sigma_{2})\subseteq\mbox{Crit}'(f).
\end{align*}
This implies that $\Sigma_{\max}$ is a $\mbox{Crit}'(f)$-exceptional
set, and is contained in $S_{\tiny{\mbox{Crit}'(f)}}$.

\medskip

Using \eqref{equ_cardnality}, we have
$$
\sup_{f\in\sA_{n}}\sharp\Sigma_{\max}\leq \sup_{f\in\sA_{n}}\sharp
S_{\tiny{\mbox{Crit}'(f)}}\leq 3\sharp\mbox{Crit}'(f)+4=3n+4,
$$
as wanted.
\end{proof}

\medskip

\begin{rema}\label{rem_lowerbound}
Other than the upper bound estimation, we also have
\begin{equation}\label{equ_lowerbound}
3n-1\leq\sup_{f\in\sA_{n}}\sharp\Sigma_{\max}.
\end{equation}
The proof is actually based on the construction of a real polynomial
associated by an admissible kneading data. Since this result is not
used in the sequel, we omit the details on the proof.
\end{rema}

\begin{rema}\label{rem_parallel}
In the view of estimations \eqref{equ_universal bound} and
\eqref{equ_lowerbound}, Corollary \ref{coro_universal bound} is a
parallel property to that in the complex setting: every exceptional
set of a rational function has at most 4 elements.
\end{rema}

\section{Co-homology}\label{sec_exceptional}
This section is devoted to the co-homologous transformation of
$\log|Df|$ as stated in Proposition \ref{lem_cohomology} below. The
idea originates from \cite{MakSmi96,MakSmi00}, and the construction
is analogous to the constructions of the ramification function of
Thurston mappings \cite{DouHub93}, (but in an inverse direction). On
the other side, some special cases of Proposition
\ref{lem_cohomology} (e.g., Chebyshev Polynomials) are also
discussed in \cite{BK98}.

To state Proposition \ref{lem_cohomology}, we recall $\cU$ the set
of upper semi-continuous potentials in
\eqref{equ_uppersemicontinuous1}, and for each $u\in\cU$, recall
$\Lambda(u)$ the singular set defined in \eqref{equ_singluar set}.
Given a map $f\in\sA$, recall $\Sigma_{\max}$ the maximum
exceptional set in \eqref{equ_maximum exceptional set}. We highlight
that the proof of Proposition \ref{lem_cohomology} is more
complicated than that for complex ration maps. This is due to the
fact that the topological structure of $\Sigma_{\max}$ for interval
maps is highly more sophisticated than that for the complex rational
maps.


\begin{prop}\label{lem_cohomology}
Let $f:J(f)\to J(f)$ be an interval map in $\sA$, then there exists
a lower semi-continuous map $h:J(f)\to\R\cup\{+\infty\}$, such that
\begin{enumerate}
  \item[$(1)$] $h$ only has log poles in $\Sigma_{\max}$;
  \item[$(2)$] Let $G:=\log|Df|+h\circ f-h$, then $G\in\cU$.
\end{enumerate}
Moreover,
\begin{enumerate}
  \item \begin{equation}\label{equ_cohomologyintegral}
  \widetilde{P}(f,-tG)=\widetilde{P}(t),~~\forall t<0,
  \end{equation}
  \item The singular set $\Lambda(G)$ is a proper subset of $\mbox{Crit}'(f)$,
  and $G$ is non-exceptional. In other words,
  there is no forward invariant finite set $\Sigma\subset J(f)$
  satisfying
  \begin{equation}\label{equ_nonexceptional property}
  f^{-1}(\Sigma)\backslash\Sigma\subseteq\Lambda(G).
  \end{equation}
\end{enumerate}
\end{prop}

\medskip

\begin{proof}[Proof of Proposition \ref{lem_cohomology}]
We split the proof into 3 parts. In part 1, we give the concrete
formalism of $h$. In part 2, we show that the new potential
$G\in\cU$. Finally, we prove statements $(a)$ and $(b)$ in Part 3.

{\bf 1.} If $\Sigma_{\max}=\emptyset,$ then $G:=\log|Df|$ has
already satisfies the desired assertions $(a)$ and $(b)$. In this
setting, $h=0$ and there is nothing to prove. So, without loss of
generality, we assume that $\Sigma_{\max}\neq\emptyset$. By
Corollary \ref{coro_universal bound}, it follows that
$\Sigma_{\max}$ is also a $\mbox{Crit}'(f)$-exceptional set, and
$\sharp\Sigma_{\max}\leq 3\sharp\mbox{Crit}'(f)+4$. Denote by
$\mbox{Crit}^*(\Sigma_{\max}):=f^{-1}(\Sigma_{\max})\backslash\Sigma_{\max}$,
and thus $\mbox{Crit}^*(\Sigma_{\max})$ is a subset contained in
$\mbox{Crit}'(f)$.


To define the function $h$, we need the following definitions. For
each $x\in J(f)$, put
\begin{equation}\label{equ_local degree}
\ell(x):=\left\{\begin{array}{cc}
            \ell_{x} & ~\mbox{if}~x\in\mbox{Crit}'(f);\\
            1 & \mbox{otherwise}.
          \end{array}\right.
\end{equation}
On the other hand, for each critical point
$\xi\in\mbox{Crit}^{*}(\Sigma_{\max}),$ put $n(\xi)$ be the minimal
non-negative integer $n$ such that $f^{n}(\xi)$ is periodic. For
each periodic cycle $\cO\subset\Sigma_{\max}$, put
\begin{equation}\label{equ_cyclecoefficients}
\hat{\alpha}(\cO):=\max\left\{\prod_{j=0}^{n(\xi)-1}(\ell(f^{j}(\xi)))^{-1}:\xi\in\mbox{Crit}^{*}(\Sigma_{\max})~\mbox{and}~f^{n(\xi)}(\xi)\in
\cO\right\}.
\end{equation}
Based on \eqref{equ_local degree} and \eqref{equ_cyclecoefficients},
we define a map $\alpha:J(f)\rightarrow \R$ inductively by
\begin{equation}\label{equ_coefficient}
\alpha(\xi):=\left\{\begin{array}{ll}
          \hat{\alpha}(\cO)-1,&\mbox{if}~~\xi\in\cO\cap\Sigma_{\max};\\
          (\alpha(f(\xi))+1)\ell(\xi)-1,& \mbox{if}~~\xi~\mbox{nonperiodic
          in}~\Sigma_{\max};\\
          0 &\mbox{Otherwise}.
         \end{array}
          \right.
\end{equation}
With this convention, put
\begin{equation}\label{equ_hx}
h(x):=\sum_{\xi\in\Sigma_{\max}}\alpha(\xi)\log|x-\xi|,~~\forall
x\in J(f).
\end{equation}
By the construction, $h$ only has log poles in $\Sigma_{\max}$. In
the rest of Part 1, we show the lower semi-continuity of $h$. This
is sufficient to show that
\begin{equation}\label{equ_degreecohomology}
-1<\alpha(\xi)\leq0,~\forall\xi\in\Sigma_{\max}.
\end{equation}
There are two cases.
\begin{enumerate}
  \item[{\bf Case1:}]~$\xi$ is not periodic. Then there exist a
  $z\in\mbox{Crit}^*(\Sigma_{\max})$ and an integer $1\leq i\leq
  n(z)-1$ such that $\xi=f^{i}(z)$. Therefore,
\begin{alignat*}{2}
0&<\alpha(\xi)+1=\left[\prod_{j=0}^{n(z)-2-i}\ell(f^{j}(\xi))\right]\cdot\hat{\alpha}(\cO)\ell(f^{n(z)-1-i}(\xi))\\
&=\prod_{j=0}^{n(z)-1-i}\ell(f^{j}(\xi))\cdot\hat{\alpha}(\cO)\\
&\leq
\prod_{j=0}^{i-1}(\ell(f^{j}(z)))^{-1}~~~&\text{(Using \eqref{equ_coefficient})}\\
&\leq 1.~~~~~&\text{(Using non-flatness hypothesis)}
\end{alignat*}
This directly implies $-1<\alpha(\xi)\leq0$.
  \item[{\bf Case2:}]~$\xi$ is periodic. Then $-1<\alpha(\xi)=\hat{\alpha}(\cO)-1\leq 0.$
\end{enumerate}
In both cases, we obtain the desired inequality
\eqref{equ_degreecohomology}, and thus $h$ is lower semi-continuous.

\medskip

{\bf 2.} Put
\begin{equation}\label{equ_cohomologouspotential}
G:=\log|Df|+h\circ f-h.
\end{equation}
We show that $G\in\cU$ in this part. In fact, the non-flatness
hypothesis yields that for each $\xi\in f^{-1}(\Sigma_{\max})$,
there is a H\"{o}lder continuous map $t_{\xi}(x)$ with
$$
|f(x)-f(\xi)|=t_{\xi}(x)|x-\xi|^{\ell(\xi)},~\mbox{and}~\inf_{\xi\in
f^{-1}(\Sigma_{\max})}\inf_{x\in~J(f)}t_{\xi}(x)>0.
$$
Hence, there is also a H\"{o}lder continuous function
$t(x)>0,~~\forall x\in ~J(f)$ with
$$
|Df(x)|=\prod_{c\in\tiny{\mbox{Crit}'}(f)}|x-c|^{\ell(c)-1}t(x).
$$

Therefore,
\begin{align*}
h\circ f(x)&=\sum_{\xi\in\Sigma_{\max}}\alpha(\xi)\log|f(x)-\xi|\\
&=\sum_{\xi\in f^{-1}(\Sigma_{\max})}\alpha(f(\xi))\log|f(x)-f(\xi)|\\
&=\sum_{\xi\in
f^{-1}(\Sigma_{\max})}\alpha(f(\xi))\log\big(|x-\xi|^{\ell(\xi)}\cdot
t_{\xi}(x)\big)\\
&=\sum_{\xi\in
f^{-1}(\Sigma_{\max})}\alpha(f(\xi))\big[\ell(\xi)\cdot\log|x-\xi|+\log
t_{\xi}(x)\big],
\end{align*}
and thus
\begin{align*}
G(x)&=\left(\sum_{\xi\in\tiny{\mbox{Crit}}'(f)}(\ell(\xi)-1)\log|x-\xi|+\log t(x)\right)\\
&+\left(\sum_{\xi\in
f^{-1}(\Sigma_{\max})}\alpha(f(\xi))\left[\ell(\xi)\cdot\log|x-\xi|+\log
t_{\xi}(x)\right]\right)-\left(\sum_{\xi\in\Sigma_{\max}}\log\alpha(\xi)|x-\xi|\right)\\
&=\left(\log~t(x)+\sum_{\xi\in f^{-1}(\Sigma_{\max})}\alpha(f(\xi))\log t_{\xi}(x)\right)\\
&+\left(\sum_{\xi\in
f^{-1}(\Sigma_{\max})}[(\ell(\xi)-1)+\alpha(f(\xi))\ell(\xi)-\alpha(\xi)]\log|x-\xi|\right)\\
&+\left(\sum_{\xi\in\tiny{\mbox{Crit}}'(f)\backslash
f^{-1}(\Sigma_{\max})}(\ell(\xi)-1)\log|x-\xi|\right)\\
&\equiv g(x)+\sum_{\xi\in
f^{-1}(\Sigma_{\max})}b(\xi)\log|x-\xi|+\sum_{\xi\in\tiny{\mbox{Crit}'(f)}\backslash
f^{-1}(\Sigma_{\max})}(\ell(\xi)-1)\log|x-\xi|,
\end{align*}
with $g(x):=\log~t(x)+\sum_{\xi\in
f^{-1}(\Sigma_{\max})}\alpha(f(\xi))\log t_{\xi}(x)$ and
$b(\xi):=(\ell(\xi)-1)+\alpha(f(\xi))\ell(\xi)-\alpha(\xi),~~\forall
\xi\in f^{-1}(\Sigma_{\max})$.

\medskip

In the view of the formalism above, the H\"{o}lder continuality of
$g$ is obvious, and the non-flat hypothesis ensures
$\ell(\xi)-1\geq0,~~\forall\xi\in\mbox{Crit}'(f)\backslash
f^{-1}(\Sigma_{\max})$. Therefore, in order to show $G\in\cU$, it is
sufficient to show that
\begin{equation}\label{equ_index2}
b(\xi)\geq0,~~\forall \xi\in
\mbox{Crit}^{*}(\Sigma_{\max}),~~\mbox{and}~~b(\xi)=0,~\forall
\xi\in\Sigma_{\max}.
\end{equation}
There are three cases.

\begin{enumerate}
  \item[{\bf Case 1:}]~$\xi\in\mbox{Crit}^{*}(\Sigma_{\max}).$ Then for each $0\leq
i\leq n(\xi)-2$, let $z:=f^{n(\xi)-1}(\xi)\in f^{-1}(\cO)$, it then
follows from \eqref{equ_cyclecoefficients} and
\eqref{equ_coefficient} that
\begin{align*}
\alpha(f^{n(\xi)-(i+1)}(\xi))&=\left[\hat{\alpha}(\cO)\cdot\prod_{j=0}^{i}\ell(f^{-i+j}(z))\right]-1\\
&\geq\left[\prod_{j=0}^{n(\xi)-1}(\ell(f^{j}(\xi)))^{-1}\cdot\prod_{j=0}^{i}\ell(f^{-i+j}(z))\right]-1\\
&\geq\left[\prod_{j=0}^{n(\xi)-1-(i+1)}(\ell(f^{j}(\xi)))^{-1}\right]-1\\
&=\left[\prod_{j=0}^{n(\xi)-i-2}(\ell(f^{j}(\xi)))^{-1}\right]-1.
\end{align*}
Hence when $i=n(\xi)-2$, we have $\alpha(f(\xi))\geq
\ell(\xi)^{-1}-1$. This means $
b(\xi)\equiv\ell(\xi)(\alpha(f(\xi))+1)-(\alpha(\xi)+1) \geq
\ell(\xi)(\ell(\xi)^{-1}-1+1)-(0+1)=0. $

\medskip

\item[{\bf Case 2:}]~$\xi\in\Sigma_{\max}$, and $\xi$ is not periodic. Then
\eqref{equ_coefficient} directly implies that
$$
b(\xi)=\ell(\xi)(\alpha(f(\xi))+1)-(\alpha(\xi)+1)=0
$$

\medskip

\item[{\bf Case 3:}]~$\xi$ is periodic. Then $\ell(\xi)=1$, and
\eqref{equ_coefficient} implies that
$$
b(\xi)=\ell(\xi)(\alpha(f(\xi))+1)-(\alpha(\xi)+1)=\hat{\alpha}(\cO)-\hat{\alpha}(\cO)=0.
$$
\end{enumerate}
As a conclusion, we obtain \eqref{equ_index2} for all cases above.
Thus, the new potential $G\in\cU$.

\medskip

{\bf 3.} In this part, we verify the statements $(a)$ and $(b)$.
From Part 1, it is clear that $h$ is finite outside $\Sigma_{\max}$.
Thus, for each $t<0$ and $\mu\in\widetilde{\cM}(f,J(f))$, we have
$$
\int_{J(f)}-tGd\mu=\int_{J(f)}-t\log|Df|d\mu,
$$
This directly yields $\widetilde{P}(f,-tG)=\widetilde{P}(t)$, which
completes the proof of Statement $(a)$.

Next, we verify Statement $(b)$. On one hand, based on the
estimations above, we have $b(\xi)=0$ for every periodic point
$\xi\in\Sigma_{\max}$. On the other side, the maximality in the
definition of $\hat{\alpha}(\cO)$ yields that for each periodic
cycle $\cO\subset\Sigma_{\max}$, there is at least a critical point
$\xi\in\mbox{Crit}^{*}(\Sigma_{\max})$ with $f^{n(\xi)}(\xi)\in\cO$
and
$$
\alpha(f(\xi))=\ell(\xi)^{-1}-1,
$$
this implies
\begin{equation}\label{equ_modified coefficient}
b(f^{j}(\xi))=0,~~\forall j\in\{0,1,\cdots,n(\xi)-1\}.
\end{equation}
Thus, $\Lambda(G)$ is a proper subset of $\mbox{Crit}'(f)$. By the
maximality of $\Sigma_{\max}$, it is easy to see that each critical
point $\xi\in\mbox{Crit}^*(\Sigma_{\max})$ is
$\mbox{Crit}'(f)$-normal, so $\xi$ is $\Lambda(G)$-normal.
Therefore, for every finite subset $\Sigma'\subseteq J(f)$
satisfying $f(\Sigma')\subseteq \Sigma'$, the potential $G$ is
finite at some points in ~$f^{-1}(\Sigma')\backslash\Sigma'$. Thus
we obtain the desired \eqref{equ_nonexceptional property} and hence
Statement $(b)$. The proof of Proposition \ref{lem_cohomology} is
thus completed.
\end{proof}

\bigskip

\section{Iterated multi-valued function systems}\label{sec_IMFS}
In this section, we construct an ``Iterated multi-valued function
system''. This is the main ingredient in the proof of Statement
$(b)$ of Key Lemma, and is stated as Proposition \ref{prop_2new}
below.
\subsection{Iterated Multi-valued Function Systems} The machinery of
``Iterated Multi-valued Function Systems'' approach is introduced in
\cite{LiRiv13a}, motivated by dealing with the situation where the
invariant measure has zero Lyapunov exponent. It is a generalization
of \cite[Main theorem]{InoRiv12}, and is based on a more general
type of induced systems.

We follow up notations from \cite[\S3]{LiRiv13a}. Given an interval
map $f\in\sA$, a compact and connected subset $B_{0}$\footnote{We
mean there is a compact and connected subset $B'_{0}$ in $I$, such
that $B_{0}=B'_{0}\cap J(f)$.} of $J(f)$, a sequence of multi-valued
functions \footnote{A \emph{multi-valued function} $\phi:B\to W$ is
a function which maps each point $x$ of $B$ to a non-empty subset
$\phi(x)$ of $W$.} $(\phi_{l})_{l=1}^{+\infty}$ is an \emph{Iterated
Multi-valued Function System (IMFS) generated by $f$}, if for each
integer $l\geq1$, there exist an integer $m_{l}\geq0$ and a pull
back\footnote{For a subset $V\subset J(f)$ and an integer $m\geq1$,
each connected component of $f^{-m}(V)$ is a \emph{pull-back} of $V$
by $f^{m}$.} $W_{m_l}$ of $B_{0}$ by $f^{m_{l}}$ contained in
$B_{0}$, such that
\begin{itemize}
  \item $f^{m_{l}}$ has an \emph{onto property} from $W_{m_l}$ to
  $B_{0}$, i.e., $f^{m_{l}}(W_{m_l})=B_{0}$;
  \item $\phi_{l}=(f^{m_{l}}|_{W_{m_l}})^{-1}.$
\end{itemize}
With this convention, we say $(\phi_{l})_{l=1}^{\infty}$ is defined
on $B_{0}$, with $(m_{l})_{l=1}^{+\infty}$ as its \emph{time
sequence}.

Let $(\phi_{l})_{l=1}^{+\infty}$ be an IMFS generated by $f$ defined
on $B_{0}$ with time sequence $(m_{l})_{l=1}^{+\infty}$. For each
integer $n\geq 1$, put $\Omega_{n}:=\{1,2\cdots\}^n$, and denote the
space of all finite words in the alphabet $\{1,2\cdots\}$ by
$\Omega^*:=\bigcup_{n\geq1}\Omega_{n}$. For every integer $k\geq1$
and every $\underline{l}:=l_{1}l_{2}\cdots l_{k}\in\Omega^*$, put
$$
|\underline{l}|=k,~~m_{\underline{l}}=m_{l_1}+m_{l_2}+\cdots+m_{l_{k}},~
\mbox{and}~~\phi_{\underline{l}}=\phi_{l_1}\circ\phi_{l_2}\circ\cdots\phi_{l_{k}}.
$$
Note that for every $x_{0}\in B_{0}$, and every pair of distinct
words $\underline{l}$ and $\underline{l}'$ in $\Omega^*$ satisfying
$m_{\underline{l}}=m_{\underline{l}'}$, we have:
\begin{equation}\label{equ_IMFS}
    \mbox{If the set $\phi_{\underline{l}}(x_{0})$ and $\phi_{\underline{l}'}(x_{0})$ intersect, then
    $\phi_{\underline{l}}(x_{0})=\phi_{\underline{l}'}(x_{0})$.}
\end{equation}
The IMFS is said to be \emph{free}, if there is an $x_{0}\in B_{0}$
such that for every pair of distinct words $\underline{l}$ and
$\underline{l}'$ in $\Omega^*$ with
$m_{\underline{l}}=m_{\underline{l}'}$, the set
$\phi_{\underline{l}}(x_{0})$ and $\phi_{\underline{l}'}(x_{0})$ are
disjoint.

\medskip

\begin{prop}\label{prop_2new}
Let $f:J(f)\to J(f)$ be an interval map in $\sA$, and $G$ be the
resulting upper semi-continuous potential in $\cU$ given by
Proposition \ref{lem_cohomology}. Let $\mu$ be an ergodic measure in
$\widetilde{\cM}(f,J(f))$ with positive Lyapunov exponent. Then
there exists a subset $X$ of $J(f)$ of full measure with respect to
$\mu$, such that for every point $x_{0} \in X$ and $t<0$, the
following property holds: There exist a constant $C>0$, a compact
and connect subset $B_{0}$ of $J(f)$ containing $x_{0}$, and a free
IMFS $(\phi_{l})_{l=1}^{+\infty}$ generated by $f$ with time
sequence $(m_{l})_{l=1}^{+\infty}$, such that
$(\phi_{l})_{l=1}^{+\infty}$ is defined on $B_{0}$, and such that
for every integer $l\geq1$ and every $z\in\phi_{l}(B_{0})$, we have
\begin{equation}\label{equ_sum}
   S_{m_{l}}(-tG)(z)\geq
   m_{l}\int_{J(f)}-tGd\mu-C.
\end{equation}
\end{prop}
\begin{rema}\label{rem_obstacles}
Analogous to that in the proof of \cite[Main Theorem]{InoRiv12}, the
main step in the proof of Proposition \ref{prop_2new} is the
construction of an IMFS, based on a given invariant measure with
strictly positive Lyapunov exponent. However, since the potential
$G$ take value $-\infty$ at the singular set $\Lambda(G)$, we need
to modify the construction of the IMFS to bypass $\Lambda(G)$ (so
that we can obtain the constant $C$ in \eqref{equ_sum}). In fact,
the modification is closed related to Proposition
\ref{lem_postercritical} developed in \S\ref{sec_normality}. On the
other hand, the IMFS is also a powerful tool on overcoming the
difficulties (e.g., discovering an appropriate pull back and time
sequence, on which the IMFS admits the onto property) arising from
the fact that interval maps are not open maps in general.
\end{rema}

\begin{rema}
Although Proposition \ref{prop_2new} is written for interval maps in
$\sA$, it also works well for the complex rational maps of degree at
least 2 on the Riemann sphere, acting on its Julia set. In fact, the
proof will be even simpler since rational maps are open.
\end{rema}

\medskip

The proof of Proposition \ref{prop_2new} depends on several lemmas
and will be given in the end of this section.



\subsection{Pesin's theory}
Let us begin with some preliminaries related to Pesin's theory.
Recall first the definition of the natural extension of $f$ on
$J(f)$. Let $\Z_{-}$ denote the set of all non-positive integers and
endow
$$
\mathcal{Z}:=\{(z_{n})_{n\in\Z_{-}}\in(J(f))^{\Z_{-}}:~~\mbox{for
every}~n\in\Z_{-},f(z_{n-1})=z_{n}\}
$$
with the product topology. Define $F:\mathcal{Z}\to\mathcal{Z}$ by
$$
F((\cdots,z_{-2},z_{-1},z_{0}))=(\cdots,z_{-2},z_{-1},z_{0},f(z_{0})),
$$
and $\pi:\mathcal{Z}\to ~J(f)$ by $\pi((z_{n}))_{n\in\Z_{-}}=z_{0}$.
If $\mu$ is a Borel probability measure that is invariant and
ergodic for $f$, then there exists a unique Borel probability
measure $\nu$ on $\mathcal{Z}$ that is invariant and ergodic for
$F$, and satisfies $\pi_{*}\nu=\mu.$ We say $(\mathcal{Z},F,\nu)$ is
the \emph{natural extension} of $(J(f),f,\mu)$.

\medskip

Pointwise ergodic theorem yields the following lemma.
\begin{lemm}\cite[Lem1.3]{PRLS04}\label{lem_pointwiseergodic}
Let $(\cZ,\nu)$ be a probability space, and let $F:\cZ\to\cZ$ be an
ergodic measure preserving transformation. Then for each function
$\phi:\cZ\to\R$ that is integrable with respect to $\nu$, there
exists a subset $Z$ of $\cZ$ such that $\nu(Z)=1$, and such that for
every $z\in Z$, we have
$$
\limsup_{n\to\infty}\sum_{i=0}^{n-1}\left(\phi(F^{i}(z))-\int_{\cZ}\phi
d\nu\right)\geq0.
$$
\end{lemm}

We also need a version of Ledrappier's unstable manifold theorem
from Dobbs.
\begin{lemm}\label{lem_dobbs}\cite{Dob13a}
Let $f$ be an interval map in $\sA$. Suppose measure
$\mu\in\cM(f,J(f))$ is ergodic and has $\chi_{\mu}>0$. Denote by
$(\cZ, F, \nu)$ the natural extension of $(J(f),f,\mu)$. Then there
exists a measurable function $\alpha$ on $\cZ$ such that
$0<\alpha<1/2$ almost everywhere with respect to $\nu$, and such
that for $\nu-$almost every point $y\in\cZ$ there exists a set
$V_{y}$ contained in $\cZ$ with the following properties:
\begin{enumerate}
  \item $y\in V_{y}$ and $\pi(V_{y})=B(\pi(y),\alpha(y))$;
  \item For each integer $n\geq0$, $f^{n}:\pi(F^{-n}(V_{y}))\to
  V_{y}$ is diffeomorphic;
  \item For each $y'\in V_{y}$,
  $$\sum_{i=0}^{+\infty}|\log|Df(\pi(F^{-i}(y)))|-\log|Df(\pi(F^{-i}(y')))||<\log2.$$
  \item For each $\eta>0$ there is a measurable function $\theta$ on
  $\cZ$ with $0<\theta<+\infty$ almost everywhere with respect to
  $\nu$, such that
  $$
  \frac{1}{\theta(y)}\exp(n(\chi_{\nu}-\eta))\leq|(Df^{n})(\pi(F^{-n}(y)))|\leq\theta(y)\exp(n(\chi_{\nu}+\eta)).
  $$
  In particular,
  $$
  |\pi(F^{-n}(V_{y}))|\leq2\theta(y)\exp(-n(\chi_{\nu}-\eta)).
  $$
\end{enumerate}
\end{lemm}
\begin{rema}\label{rem_pesintheory}
A stronger version (with the same proof) of Property $(c)$ is
possible: For each upper semi continuous potential $u\in\cU$, there
is a constant $\widetilde{C}$ such that for each $y'\in V_{y}$, we
have
$$
\sum_{i=0}^{\infty}\left|u(\pi(F^{-i})(y))-u(\pi(F^{-i}(y')))\right|<\widetilde{C}.
$$
\end{rema}

\medskip

The following Lemma follows from the former lemma using known
arguments. We put the detailed proof for completeness.
\begin{lemm}\label{lem_huaibin2}
Let $f:J(f)\to J(f)$ be an interval map in $\sA$, and
$u:J(f)\to\R\cup\{-\infty\}$ be a upper semi continuous potential in
$\cU$. Let $\mu$ be an ergodic measure in $\cM(f, J(f))$ with
$\chi_{\mu}>0$. Then there is a subset $X^*$ of $~J(f)$ full measure
with respect to $\mu$ possessing the following property: For every
point $x\in X^*$ and every $t<0$, there exist $\bar{\rho}_{x}>0$,
$D>0$ and a strictly increasing sequence of positive integers
$(n_{l})_{l=1}^{+\infty}$, such that for every $l>0$, we can choose
a point $x_{n_{l}}\in f^{-n_{l}}(x)$ and a connected component
$W_{n_l}$ of $f^{-n_{l}}(B(x,\bar{\rho}_{x}))$ containing
$x_{n_{l}}$ so that:
\begin{enumerate}
  \item $x_{n_{l+1}}\in f^{-(n_{l+1}-n_{l})}(x_{n_l})$;
  \item For every point $y\in W_{n_l}$,
  \begin{equation}\label{equ_pesinsequence2}
  S_{n_{l}}(-tu)(y)\geq n_{l}\int_{J(f)}-tud\mu-D;
  \end{equation}
  \item $\lim_{l\to\infty}|W_{n_l}|=0.$
\end{enumerate}
\end{lemm}
\begin{proof}[Proof of Lemma \ref{lem_huaibin2}]
Let $(\cZ,F,\nu)$ be the natural extension of $(J(f),f,\mu)$, and
thus $\nu$ is also invariant and ergodic with respect to $F^{-1}$.
For each $t<0$, by applying Lemma \ref{lem_pointwiseergodic} for
$F^{-1}$ to the integrable function\footnote{The integrability
directly follows from the upper semi-continuity of $-tu\circ \pi$,
with $t<0$.} $-tu\circ\pi$, there exists a subset $Z$ of $\cZ$ of
full measure with respect to $\nu$, such that
\begin{equation}\label{equ_limsupsequnce}
    \limsup_{n\to\infty}\sum_{n=0}^{n-1}\left(-tu\circ
\pi(F^{-i}(z_{m})_{m\in\Z_{-}})+\int_{\cZ}tu\circ \pi
d\nu\right)\geq0,~~\forall (z_{m})_{m\in\Z_{-}}\in Z.
\end{equation}
Taking a subset of $Z$ of full measure with respect to $\nu$ if
necessary, by Lemma~\ref{lem_dobbs}, there is a function
$\alpha:\cZ\to(0,1/2)$ such that $Z$ and $\alpha$ satisfy all the
assertions of Lemma \ref{lem_dobbs}. Define the set $X^*:=\pi(Z)$,
then we have
$$
\mu(X^{*})=\nu(\pi^{-1}(\pi(Z)))\geq \nu(Z)=1.
$$

\medskip

In the rest of the proof, we will verify that $X^{*}$ satisfies the
desired properties. Fix a point $x\in X^*$, and choose a point
$(z_{m})_{m\in\Z_{-}}$ such that $\pi((z_{m})_{m\in\Z_{-}})=x$, let
$V_{(z_{m})_{m\in\Z_{-}}}$ be given by Lemma \ref{lem_dobbs} for the
point $(z_{m})_{m\in\Z_{-}}$, and put
$\rho_{x}:=\alpha((z_{m})_{m\in\Z_{-}})$. Moreover, for each integer
$j\geq1,$ put
$$
y_{j}:=\pi(F^{-j}(z_{m})_{m\in\Z_{-}})=z_{j}\in f^{-j}(x),
$$
and $U_{j}:=\pi(F^{-j}(V_{(z_{m})_{m\in\Z_{-}}})).$ Using Assertions
$(a)$ and $(b)$ of Lemma \ref{lem_dobbs}, it follows that for every
integer $j\geq1$, $U_{j}$ is the connect component of
$f^{j}(B(x,\rho_{x}))$ containing $y_{j}$, and $f^{j}:U_{j}\to
B(x,\rho_{x})$ is diffeomorphic. Moreover, using Remark
\ref{rem_pesintheory} and part $(4)$ of Lemma \ref{lem_dobbs}, it
follows that there exist $C'>0$ and $\lambda>1$, such that every
$n\geq1$, we have
\begin{equation}\label{equ_expshrinking}
    |U_{n}|\leq C'\lambda^{-n},
\end{equation}
and there is a constant $\widetilde{C}>0$, such that for every two
points $x,y\in U_{n}$, we have
\begin{equation}\label{equ_bounded distortion1}
|S_{n}(-tu)(x)-S_{n}(-tu)(y)|<-t\widetilde{C},~~\forall t<0.
\end{equation}

Finally, fix a $D'>0$, the inequality \eqref{equ_limsupsequnce}
yields that there is a strictly increasing sequence of positive
integers $(n_{l})_{l=1}^{+\infty}$ such that for every $l\geq0$, we
have
\begin{equation}\label{equ_increasing sequence}
    \sum_{i=0}^{n_{l}-1}-tu\circ\pi(F^{-i}(z_{m})_{m\in\Z_{-}})\geq
    n_{l}\int_{\cZ}-tu\circ \pi
    d\nu-D'=n_{l}\int_{~J}-tud\mu-D'.
\end{equation}

For each integer $l\geq1$, put $x_{n_{l}}:=y_{n_{l}}$ and
$W_{n_l}:=U_{n_{l}}$, then statement $(a)$ and $(c)$ are
automatically satisfied from the definition of $y_{n_{l}}$ and
inequality \eqref{equ_expshrinking} respectively, and Statement
$(b)$ follows from \eqref{equ_bounded distortion1} and
\eqref{equ_increasing sequence} with $D:=D'-t\widetilde{C}$. The
proof of this lemma is thus completed.
\end{proof}

The following three technical lemmas are also required.

\begin{lemm}\cite[Lem3.2]{LiRiv13a}\label{lem_boundary}
Given an interval map $f:J(f)\to J(f)$ in $\sA$, there is an
$\epsilon>0$ such that the following property holds. Let $J_{0}$ be
an interval contained in $I$ satisfying $|J_{0}|\leq\varepsilon$,
let $n\geq1$ be an integer, and let $\hat{J}$ be a pull-back of
$J_{0}$ by $f^{n}$, whose closure is contained in the interior of
$I$. Suppose in addition that for each $j\in\{1,\cdots,n\}$ the
pull-back of $J_{0}$ by $f^{j}$ containing in $f^{n-j}(\hat{J})$ has
length bounded from above by $\varepsilon$. Then $f^{n}(\partial
\hat{J})\subset
\partial J_{0}.$
\end{lemm}

\medskip

\begin{lemm}\cite[Lem3.3]{LiRiv13a}\label{lem_oneside}
Let $f:J(f)\to J(f)$ be an interval map in $\sA$, and let $a$ be a
point of $J(f)$ such that $(a,+\infty)$(resp. $(-\infty,a)$)
intersects $J(f)$. Then for every open interval $U$ intersecting
$J(f)$, and every sufficient large integer $n\geq1$, there is a
point $y$ of $U$ in $f^{-n}(a)$ such that for every $\epsilon>0$,
the set $f^{n}(B(y,\epsilon))$ intersects
$(a,+\infty)$(resp.$(-\infty,a)$).
\end{lemm}

\medskip

\begin{lemm}\cite[Lem A.2]{Riv13b}\label{lem_pullback} Given an interval map
$f:J(f)\to J(f)$ in $\sA$, then for every $\kappa>0$, there is a
$\delta>0$ such that for every $x\in J(f)$, every integer $n\geq1$,
and every pull-back $W$ of $B(x,\delta)$ by $f^{n}$, we have
$|W|<\kappa.$
\end{lemm}

\medskip

\subsection{Proof of Proposition \ref{prop_2new}}
We are ready to prove Proposition \ref{prop_2new} in this
subsection. The proof is divided into two parts. In Part $(1)$ below
we define a IMFS concretely. In Part $(2)$, we show the IMFS is free
and satisfies \eqref{equ_sum}.

\begin{proof}
Let $\epsilon>0$ be the constant given by Lemma \ref{lem_boundary}
and let $\delta>0$ be the constant given by Lemma \ref{lem_pullback}
for $\kappa=\varepsilon$. Let $X^*$ be the subset in $J(f)$ given by
Lemma \ref{lem_huaibin2}, and put $X$ be the complement of $X^{*}$
of the set of pre-periodic points of $f$. Since the measure $\mu$ is
ergodic and not supported on a periodic orbit, the set $X$ has full
measure for $\mu$. Fix a point $x_{0}\in X$ which is not an end
point of $I$.

\medskip

{\bf 1.} Let
$\bar{\rho}_{x_{0}},(n_{l})_{l=1}^{+\infty},(x_{n_{l}})_{l=1}^{\infty}$
and $(W_{n_l})_{l=1}^{+\infty}$ be given by Lemma \ref{lem_huaibin2}
with $x=x_{0}$. Fix $
\rho\in(0,\min\{\delta,\bar{\rho}_{x_{0}},\mbox{dist}(x_{0},\partial
I)\}).$ Taking a subsequence if necessarily, and suppose
$$
\lim\limits_{l\to\infty}x_{n_{l}}=b.
$$
Using Lemma \ref{lem_huaibin2} for $u=G$, it follows that for each
$t<0$, there is a constant $D>0$, such that for every $z\in
W_{n_{l}}$, we have
\begin{equation}\label{equ_good part}
    S_{n_{l}}(-tG)(z)\geq n_{l}\int_{J(f)}-tGd\mu-D
\end{equation}
On the other side, Proposition \ref{lem_cohomology} yields that for
every non empty forward invariant finite set $\Sigma\subset J(f)$,
we have $f^{-1}(\Sigma)\backslash\Sigma\nsubseteq\Lambda(G)$.
Therefore, Proposition \ref{lem_postercritical} implies that there
exists an integer $N\geq0$ such that $b_{N}:=f^{N}(b)$ is
$\Lambda(G)$-normal and $b_{N}\notin\Lambda(G)$, and thus
\begin{equation}\label{equ_cometonormalpoint}
\lim\limits_{l\to\infty}x_{n_{l}-N}=\lim\limits_{l\to\infty}f^{N}(x_{n_{l}})=f^{N}(\lim_{l\to\infty}x_{n_{l}})=f^{N}(b)=b_{N}.
\end{equation}
So for each $\bar{\rho}>0$, the point
$x_{n_{l}-N}\in[b_{N}-\bar{\rho},b_{N}]$ for infinitely many $l$.
Meanwhile, using the normality of $b_{N}$ and the finiteness of the
critical set, we have
\begin{enumerate}
  \item there is an integer $\widetilde{M}>0$, and $b'_{N}$ such
  that $f^{\widetilde{M}}(b'_{N})=b_{N}$ and $f^{j}(b'_{N})\notin\Lambda(G),~\forall
  j=0,1,\cdots,\widetilde{M}$;
  \item every pre-image of $b'_{N}$ is not in $\mbox{Crit}'(f)$.
\end{enumerate}
By statement $(a)$, there exist $\widetilde{\rho}>0$ and a closed
one side subinterval of $B(b'_{N},\widetilde{\rho})$ (say
$(b'_{N}-\widetilde{\rho},b'_{N})$) intersecting with $J(f)$ and
satisfying
\begin{equation}\label{equ_backward1}
f^{\widetilde{M}}([b'_{N}-\widetilde{\rho},b'_{N}])=[b_{N}-\bar{\rho},b_{N}],~~
f^{j}([b'_{N}-\widetilde{\rho},b'_{N}])\cap\Lambda(G)=\emptyset.~~\forall
j=0,1,\cdots,\widetilde{M}.
\end{equation}
This implies there exist points
$x_{n_{l}-N+\widetilde{M}}\in[b'_{N}-\widetilde{\rho},b'_{N}]$ such
that $f^{\widetilde{M}}(x_{n_{l}-N+\widetilde{M}})=x_{n_{l}-N}$, for
infinitely many $l$.

\medskip

By the definition of set $X$, $x_{0}$ is not pre-periodic, so
$x_{0}$ is not in the boundary of a periodic Fatou component. Hence,
we can assume that there are two disjoint open intervals
$\widetilde{U}_{0}$ and $\widetilde{U}_{1}$ in $(x_{0}-\rho,x_{0})$,
each of them intersecting $J(f)$. Note that $f$ is topological exact
on $J(f)$, and $(b'_{N}-\widetilde{\rho},b_{N}')$ interests $J(f)$,
so by applying Lemma \ref{lem_oneside} for $a=b_{N}'$, we conclude
that there exist an integer $\hat{M}>1$, and two distinct points
$\omega_{0},\omega_{1}$ in $\widetilde{U}_{0}$ and
$\widetilde{U}_{1}$ respectively with
$f^{\hat{M}}(\omega_{0})=f^{\hat{M}}(\omega_{1})=b'_{N}$, such that
for every $\epsilon'>0$, both sets
$f^{\hat{M}}(B(\omega_{0},\epsilon'))$ and
$f^{\hat{M}}(B(\omega_{1},\epsilon'))$ intersect with
$(-\infty,b'_{N})$.

\medskip

On the other hand, due to Statement $(b)$, we can reduce the value
of $\bar{\rho}$ (and so $\widetilde{\rho}$ is reduced
correspondingly), such that the pull backs $U_{0},U_{1}$ of
$(b'_{N}-\widetilde{\rho},b'_{N})$ by $f^{\hat{M}}$ has the
properties: $\omega_{i}\in U_{i}\subset\widetilde{U}_{i},~~~\forall
i=0,1$ and
\begin{equation}\label{equ_backward2}
f^{j}(U_{i})\cap\mbox{Crit}'(f)=\emptyset,~\forall
i=0,1~\mbox{and}~~ j=0,1,\cdots,\hat{M}.
\end{equation}
So, $U_{0}$ and $U_{1}$ are disjoint and are contained in
$(x_{0}-\rho,x_{0})$, and the points $x_{n_{l}-N+\widetilde{M}}$ is
contained in both $f^{\hat{M}}(U_{0})$ and $f^{\hat{M}}(U_{1})$ for
infinitely many $l$. Together with \eqref{equ_backward1}, let
$M:=\widetilde{M}+\hat{M}$, then
\begin{equation}\label{equ_centerpoints}
\mbox{$x_{n_{l}-N}$ is contained in both $f^{M}(U_{0})$ and
$f^{M}(U_{1})$, for infinitely many $l$.}
\end{equation}

\medskip

In addition, using that $\lim_{l\to\infty}|W_{n_{l}-N}|=0$ and
taking a subsequence if necessarily, then for every $l>0$, we have
the following properties: $n_{l+1}-n_{l}\geq M$, the point
$x_{n_{l}-N}$ is contained in both $f^{M}(U_{0})$ and
$f^{M}(U_{1})$, the length $|W_{n_{l}-N}|<\varepsilon$, and the pull
back $W_{n_{l}-N}$ of $\overline{B(x_{0},\rho)}$ by $f^{n_{l}-N}$
containing $x_{n_{l}-N}$ is contained in $[b_{N}-\bar{\rho},b_{N}]$.
By interchanging $\omega_{0}$ and $\omega_{1}$ and taking a
subsequence if necessarily, we can assume that every $l$, the point
$f^{n_{l+1}-n_{l}-M+N}(x_{n_{l+1}})$ is not contained in $U_{0}$.
For each $l,$ choose the pull back $W'_{n_{l}-N+M}$ of $W_{n_{l}-N}$
by $f^{M}$ that contains a point $x'_{n_{l}-N+M}$ of
$f^{-M}(x_{n_{l}-N})$, and that is contained in $U_{0}$.

\medskip

With this convention, we have $W'_{n_{l}-N+M}\subset
U_{0}\subset(x_{0}-\rho,x_{0})$ for every $l>0$. By our choice of
$\rho$, the closure of $W'_{n_{l}-N+M}$ is contained in the interior
of $I$. On the other hand, using Lemma \ref{lem_pullback}, then the
length of $f^{i}(W'_{n_{l}-N+M})$ is less than $\varepsilon$, for
every $i=0,1,\cdots n_{l}-N+M.$ So Lemma \ref{lem_boundary} yields
that $f^{n_{l}-N+M}(\partial W'_{n_{l}-N+M})$ is contained in
$\partial B(x_{0},\rho)$. Also, note that \eqref{equ_centerpoints}
yields that $f^{n_{l}-N+M}(W'_{n_{l}-N+M})$ contains $x_{0}$, hence
we have the set $f^{n_{l}-N+M}(W'_{n_{l}-N+M})$ contains either
$[x_{0}-\rho,x_{0}]$ or $[x_{0},x_{0}+\rho]$. 
There are two probabilities.
\begin{enumerate}
\item[{\bf Case 1:}] There are infinitely many $l$ such that the
set $f^{n_{l}-N+M}(W'_{n_{l}-N+M})$ contains $[x_{0}-\rho,x_{0}]$.
We can assume this property holds for every $l$, by taking a
subsequence. Then, there is a pull back $W''_{n_{l}-N+M}$ of
$[x_{0}-\rho,x_{0}]$ by $f^{n_{l}-N+M}$ that is contained in
$W'_{n_{l}-N+M}$, and such that
$$
f^{n_{l}-N+M}(W''_{n_{l}-N+M})=[x_{0}-\rho,x_{0}].
$$
With this convention, we have
$$
W_{n_{l}-N+M}''\subseteq W'_{n_{l}-N+M}\subseteq U_{0}\subset
[x_{0}-\rho,x_{0}].
$$
Put
$$
B_{0}:=[x_{0}-\rho,x_{0}],~~M':=M,~~\mbox{and}~~U_{0}':=U_{0}.
$$

\item[{\bf Case 2:}] For every $l$ (outside finitely many exceptions), the set
$f^{n_{l}-N+M}(W'_{n_{l}-N+M})$ contains $[x_{0},x_{0}+\rho]$, but
does not contain $[x_{0}-\rho,x_{0}]$. We can assume this property
holds for every $l$, by taking a subsequence. Note that $x_{0}$ is
not in the boundary of a Fatou component, Lemma \ref{lem_oneside}
and topological exactness yield that there exist an integer
$\bar{M}\geq1$ and a pull back $U_{0}'$ of $U_{0}$ by $f^{\bar{M}}$
that is contained in $(x_{0},x_{0}+\rho)$, and such that
$x_{n_{l}-N+M}'$ is contained in $f^{\bar{M}}(U_{0}')$ for
infinitely many $l$. Take a subsequence if necessarily, assume for
every $l$, we have $n_{l+1}-n_{l}\geq M+\bar{M}$, and
$x_{n_{l}-N+M}'$ is contained in $f^{\bar{M}}(U_{0}')$.

Since for each $l$ the point $f^{n_{l+1}-n_{l}-M}(x_{n_{l+1}-N})$ is
not in $U_{0}$, it implies that the point
$f^{n_{l+1}-n_{l}-M-\bar{M}}(x_{n_{l+1}-N})$ is not in $U_{0}'.$ For
each $l$, choose a pull back $\widetilde{W}'_{n_{l}-N+M+\bar{M}}$ of
$W'_{n_{l}-N+M}$ by $f^{\bar{M}}$ contained in $U_{0}'$ and that
contains a point $\widetilde{x}'_{n_{l}-N+M+\bar{M}}$ of
$f^{-\bar{M}}(x'_{n_{l}-N+M})$. By Lemma \ref{lem_boundary}, the set
$f^{n_{l}-N+M+\bar{M}}(\partial \widetilde{W}_{l}')$ is contained in
$\partial B(x_{0},\rho)$. On the other hand, since the set
$f^{n_{l}-N+M+\bar{M}}(\widetilde{W}'_{n_{l}-N+M+\bar{M}})$ is
contained in $f^{n_{l}-N+M}(W'_{n_{l}-N+M})$. But
$f^{n_{l}-N+M}(W'_{n_{l}-N+M})$ does not contain
$[x_{0}-\rho,x_{0}]$, we conclude that
$f^{n_{l}-N+M+\bar{M}}(\widetilde{W}'_{n_{l}-N+M+\bar{M}})$ maps
both end points of $\widetilde{W}'_{n_{l}-N+M+\bar{M}}$ to the point
$x_{0}+\rho$. Note also that by the construction,
$f^{n_{l}-N+M+\bar{M}}(\widetilde{W}'_{n_{l}-N+M+\bar{M}})$ contains
the point $x_{0}$. Therefore, the set
$f^{n_{l}-N+M+\bar{M}}(\widetilde{W}'_{n_{l}-N+M+\bar{M}})$ contains
$[x_{0},x_{0}+\rho]$. So there is a pull back
$W''_{n_{l}-N+M+\bar{M}}$ of $[x_{0},x_{0}+\rho]$ by
$f^{n_{l}-N+M+\bar{M}}$ that is contained in
$\widetilde{W}'_{n_{l}-N+M+\bar{M}}$, and such that
$$
f^{n_{l}-N+M+\bar{M}}(W''_{n_{l}-N+M+\bar{M}})=[x_{0},x_{0}+\rho].
$$
With this convention, we have
$$
W''_{n_{l}-N+M+\bar{M}}\subseteq\widetilde{W}'_{n_{l}-N+M+\bar{M}}\subseteq
U'_{0}\subset [x_{0},x_{0}+\rho].
$$
Put
$$
B_{0}:=[x_{0},x_{0}+\rho],~~ M':=M+\bar{M}.
$$
\end{enumerate}

In both cases, for every $l>0$, we put
$$
\phi_{l}:=\left(f^{n_{l}-N+M'}|_{W''_{n_{l}-N+M'}}\right)^{-1}.
$$
Then $(\phi_{l})_{l=1}^{+\infty}$ is an IMFS generated by $f$, that
is defined on $B_{0}$ with the time sequence
$(m_{l})_{l=0}^{+\infty}:=(n_{l}-N+M')_{l=0}^{+\infty}$. Moreover,
$n_{l+1}-n_{l}\geq M',~W''_{n_{l}-N+M'}\subset U'_{0}$, and
$f^{n_{l+1}-n_{l}-M'}(x_{n_{l+1}-N})\notin U'_{0}$. The construction
(Case 1) of $(\phi_{l})_{l=1}^{+\infty}$ is also illustrated in
Figure \ref{figure_iternary}.

\medskip

{\bf 2.} To prove that the IMFS $(\phi_{l})_{l=1}^{+\infty}$ is
free, we follow the same idea from \cite{LiRiv13a}. Let $k,k'\geq1$
be two integers and let
$$
\underline{l}:=l_{1}l_{2}\cdots
l_{k},~~\mbox{and}~~\underline{l}':=l_{1}'l_{2}'\cdots l'_{k'}
$$
be two distinct words in $\Omega^*$ such that
$m_{\underline{l}}=m_{\underline{l}'}$. Without loss of generality
we assume that $l'_{k'}\geq l_{k}+1$, then
$$
f^{m_{\underline{l}}-m_{l_k}}(\phi_{\underline{l}}(x_{0}))=\phi_{l_{k}}(x_{0})\subset
W''_{n_{l}-N+M'}\subset U_{0}'.
$$
On the other side, we have
$$
m_{l'_{k'}}-m_{l_{k}}=n_{l'_{k'}}-n_{l_{k}}\geq
n_{l_{k}+1}-n_{l_{k}}\geq M',
$$
and thus the set
\begin{alignat*}{2}
f^{m_{\underline{l}}-m_{l_{k}}}(\phi_{\underline{l}'}(x_{0}))&=f^{m_{\underline{l}'}-m_{l_{k}}}(\phi_{\underline{l}'}(x_{0}))\\
&=f^{m_{l'_{k'}}-m_{l_{k}}}(\phi_{l_{k'}'}(x_{0}))\\
&=f^{m_{l'_{k'}}-m_{l_{k}}-M'+N}\left((f^{n_{l'_{k'}}}|_{W_{n_{l'_{k'}}}})^{-1}(x_{0})\right)
\end{alignat*}
contains the point
\begin{alignat*}{2}
f^{m_{l'_{k'}}-m_{l_{k}}-M'+N}(x_{n_{l'_{k'}}})&=f^{n_{l'_{k'}}-n_{l_{k}}-M'+N}(x_{n_{l'_{k'}}})\\
&=f^{n_{l_{k}+1}-n_{l_{k}}-M'+N}(x_{n_{l_{k}+1}}).
\end{alignat*}
By the construction, this point is not contained in $U_{0}'$, so the
sets
$$
f^{m_{\underline{l}}-m_{l_k}}(\phi_{\underline{l}}(x_{0}))
~~~\mbox{and}~~~
f^{m_{\underline{l}}-m_{l_k}}(\phi_{\underline{l}'}(x_{0}))
$$
are distinct. This implies that the sets
$\phi_{\underline{l}}(x_{0})$ and $\phi_{\underline{l}'}(x_{0})$ are
different. By \eqref{equ_IMFS}, they are actually disjoint. So the
IMFS $(\phi_{l})_{l=1}^{\infty}$ is free.

\medskip

Finally, we verify \eqref{equ_sum} below. Following from
\eqref{equ_backward1} and \eqref{equ_backward2}, we have
$$
f^{i}(W''_{n_l-N+M'})\cap\Lambda(G)=\emptyset,~~\forall
i=0,1,\cdots,M'.
$$
Hence, for each $t<0$, the numbers
$$
C_{1}:=t\sup\limits_{\varsigma\in\bigcup_{i=0}^{M'}f^{i}(W''_{n_l-N+M'})}\log|f(\varsigma)|<+\infty,
$$
and
$$
C_{2}:=-t\sup_{x\in J(f)}\log|Df(x)|<+\infty,
$$
Recall that for every $l\geq1$ and $z\in \phi_{l}(B_0)$ the point
$f^{M'-N}(z)\in W_{n_l}$. Hence,
\begin{alignat*}{2}
S_{m_{l}}(-tG)(z)
&=S_{M'-N}(-tG)(z)+S_{n_{l}}(-tG)(f^{M'-N}(z))&\\
&=S_{M'}(-tG)(z)-S_{N}(-tG)(f^{M'-N}(z))+S_{n_{l}}(-tG)(f^{M'-N}(z))&\\
&\geq-M'C_{1}+NC_{2}+n_{l}\int_{J(f)}-tGd\mu-D~~~~\text{(Using \eqref{equ_good part})}\\
&=m_{l}\int_{J(f)}-tGd\mu-D-\left(M'C_{1}-NC_{2}+(M'-N)\int_{J(f)}-tGd\mu\right).&
\end{alignat*}
This implies the desired inequality \eqref{equ_sum} with
$$
C:=D+\left(M'C_{1}-NC_{2}+(M'-N)\int_{J(f)}-tGd\mu\right)<+\infty.
$$
\end{proof}

\medskip

\section{Proof of Key Lemma}\label{sec_keylemmaperiod}
In this section, we will complete the proof of Key Lemma. In the
view of Proposition \ref{lem_cohomology}, it is sufficient to show
Statement $(b)$. We will distinguish two cases, according as the
measure with positive Lyapunov exponent is supported on a periodic
orbit or not. The former case is stated as Lemma \ref{lem_period},
and the latter case is proved in the end of this section.

Analogous to Proposition \ref{prop_2new}, Lemma \ref{lem_period} is
an adaption of \cite[Prop4.1]{InoRiv12} and
\cite[Lemm4.1]{LiRiv13b}, according to the obstacles stated in
Remark \ref{rem_obstacles}. On the other hand, we remark that Lemma
\ref{lem_period} works well also for rational maps of degree at
least 2 on the Riemann sphere.



\medskip

\begin{lemm}\label{lem_period}
Let $f:J(f)\to J(f)$ be an interval map in $\sA$, and $G$ be the
upper semi-continuous potential in $\cU$ given by Proposition
\ref{lem_cohomology}, then for every hyperbolic repelling periodic
point $x_{0}\in J(f)$ of period $N$, we have
\begin{equation}\label{equ_periodic exceptional}
    \limsup_{n\to\infty}\frac{1}{n}\log\sum_{y\in
    f^{-n}(x_{0})}\exp(S_{n}(G)(y))>\frac{1}{N}S_{N}(G)(x_{0}).
\end{equation}
\end{lemm}
\begin{proof}[Proof of Lemma \ref{lem_period}]
Analogous to that in \cite[Lem4.1]{LiRiv13a}, we split the proof
into 2 parts. In part 1, we construct the induced map, and in part 2
we show \eqref{equ_periodic exceptional} for the induced map.

\medskip

{\bf 1.} Fix a repelling periodic point $x_{0}\in~J(f)$ of period
$N$. Since $|(f^{N})'(x_{0})|>1$, there is $\rho>0$ and a local
inverse $\phi$ of $f^{2N}$ defined on $B(x_{0},\rho)$ with
$\phi(x_{0})=x_{0}$. Note that $f^{2N}\circ \phi$ is the identity
map on $B(x_{0},\rho)$. Hence $\phi'(x_{0})>0$, thus $\phi$ is
increasing on $B(x_{0},\rho)$ and $f^{2N}$ is also increasing on
$\phi(B(x_{0},\rho))$. Since $x_{0}\in J(f)$, changing orientation
and reducing $\rho$ if necessarily, we assume that
$(x_{0},x_{0}+\rho/2)$ intersects with $J(f)$. On the other hand,
Proposition \ref{lem_cohomology} implies that for every non-empty
finite forward invariant subset $\Sigma\subset J(f)$, we have
$f^{-1}(\Sigma)\backslash\Sigma\nsubseteq\Lambda(G)$. Hence, by
Proposition \ref{lem_postercritical}, $x_{0}$ itself is
$\Lambda(G)$-normal, and is not in $\Lambda(G)$. This implies:
\begin{enumerate}
  \item there exist an integer $\widetilde{k}$, a point
  $a\in~J(f)$, and an open one side subinterval (say
  $(a,a+\widetilde{\rho}/2)$) such that
  $f^{2N\widetilde{k}}(a)=x_{0},f^{2N\widetilde{k}}(a,a+\widetilde{\rho}/2)=(x_{0},x_{0}+\rho/2)$,
  and
  \begin{equation}\label{equ_nonsingular1}
f^{2Nj}(a)\notin\Lambda(G),~f^{2Nj}(a,a+\widetilde{\rho}/2)\cap\Lambda(G)=\emptyset,~~\forall
j=0,1,\cdots,\widetilde{k};
  \end{equation}
  \item every pre-image of $a$ is not in $\mbox{Crit}'(f)$.
\end{enumerate}
Note that $f$ is topological exact on $J(f)$. Using Lemma
\ref{lem_oneside} for the point $a$, it follows that there exist an
integer $\hat{k}\leq1$ and a point $z'\in (x_{0},x_{0}+\rho/2)$ such
that $f^{2N\hat{k}}(z')=x_{0}$, and such that for every
$\epsilon>0$, the set $f^{2N\hat{k}}(B(z',\epsilon))$ intersects
$(a,a+\widetilde{\rho}/2)$. Together with statement $(b)$, we can
fix $\epsilon\in(0,|z'-x_{0}|)$ such that
$f^{2N\hat{k}}(B(z',\epsilon))\subset B(a,\widetilde{\rho}/2)$, and
such that
\begin{equation}\label{equ_nonsingluar2}
f^{2Nj}(B(z',\epsilon))\cap\mbox{Crit}'(f)=\emptyset,~\forall
j=0,1,\cdots,\hat{k}.
\end{equation}
Note also the closure of $B(z',\epsilon)$ is contained in
$(x_{0},x_{0}+\rho/2)$.

Let $k':=\widetilde{k}+\hat{k}$, and $W$ be the pull back of
$f^{2Nk'}(B(z',\epsilon))\cap[x_{0},x_{0}+\rho/2)$ by $f^{2Nk'}$
containing $z'$. Put $U'_{0}:=\phi^{k'}(f^{2Nk'}(W))$. Since both
$f^{2Nk'}$ and $\phi^{k'}$ are continuous, reducing $\epsilon$ if
necessarily, $U'_{0}$ is disjoint from $\overline{W}$. By our choice
of $\phi$, it follows that
$$
W\subseteq B(z',\epsilon)\subset
(x_{0},x_{0}+\rho),~~\mbox{and}~~x_{0}\in f^{2Nk'}(W)\subseteq
[x_{0},x_{0}+\rho/2).
$$
Moreover, based on the hypothesis that $x_{0}$ is a repelling
periodic point and the definition of $\rho$, for every open set
$U\subseteq [x_{0},x_{0}+\rho/2)$, we have
$$
\lim\limits_{k\to\infty}\mbox{diam}(\phi^{k}(U))=0,~~\mbox{and}~~\lim\limits_{k\to\infty}\mbox{dist}(\phi^{k}(U),x_{0})=0.
$$
Thus, there is an integer $k_{1}\geq0$, such that
$$
U_{1}:=\phi^{k_{1}}(W)\subset f^{2Nk'}(W),
$$
and
\begin{equation}\label{equ_diameter}
\mbox{diam}(\phi^{k_{1}+k'}(f^{2Nk'}(W)))<\mbox{diam}(f^{2Nk'}(W)).
\end{equation}
Put $k_{0}:=k_{1}+k'$, and $U_{0}:=\phi^{k_{1}}(U_{0})$. Then we
have
$$
k_{0}\geq1,~U_{0}\cap U_{1}=\emptyset,~\mbox{and}~U_{1}\subset
f^{2Nk'}(W).
$$
By \eqref{equ_diameter}, and the fact that $f^{2Nk'}(W)$ contains
$x_{0}$, we have
$$
U_{0}=\phi^{k_{1}}(U_{0}')=\phi^{k_{0}}(2Nk'(W))\subset f^{2Nk'}(W).
$$
We also note that
$$
f^{2Nk_{0}}(U_{0})=f^{2Nk'}(W)=f^{2Nk_{0}}(U_{0}).
$$
Put
$$
U:=U_{0}\cup U_{1},~~\hat{f}:=f^{2Nk_{0}}|_{U}.
$$

\medskip

{\bf2.} We will prove \eqref{equ_periodic exceptional} in this part.
Put $\hat{G}:=\frac{1}{2Nk_{0}}S_{2Nk_{0}}(G)$ and for every integer
$m\geq1,$ put
$$
\hat{S}_{m}(\hat{G}):=\hat{G}+\hat{G}\circ
\hat{f}+\cdots+\hat{G}\circ\hat{f}^{m-1}.
$$
Note that to prove the Lemma, it is sufficient to show that
$$
\limsup\limits_{m\to\infty}\frac{1}{m}\log\sum_{y\in\hat{f}^{-m}(x_{0})}\exp(\hat{S}_{m}(\hat{G})(y))>\hat{G}(y).
$$
This is equivalent to show that the convergent radius of the series
\begin{equation}\label{equ_series2}
\Xi(s):=\sum_{n=0}^{+\infty}\left(\sum_{z\in\hat{f}^{n}(x_{0})}\exp\left(\hat{S}_{n}(\hat{G}(z))\right)\right)s^{n}
\end{equation}
is strictly less than $\exp(-\hat{G}(x_{0}))$.

\medskip

The proof of this fact above is analogous to the proof of
\cite[Prop4.1]{InoRiv12}. We include it for completeness.

Put $\hat{K}:=\bigcap_{i=0}^{+\infty}\hat{f}^{-i}(U)$, and consider
the itinerary map $\iota:\hat{K}\to\{0,1\}^{\N}$ defined so that for
every $i\in\{1,2\cdots\}$, the point $\hat{f}^{i}(z)$ is in
$U_{\iota(z)_{i}}$. Since $\hat{f}$ maps each of the sets
$U_{0},U_{1}$ onto $f^{Nk'}(W)$, and both $U_{0},U_{1}$ are
contained in this set, for every integer $k\geq0$, and every
sequence $a_{0},a_{1},\cdots,a_{k}$ of elements of $\{0,1\}$, there
is a point of $\hat{f}^{-k+1}(x_{0})$ in the set
$$
\hat{K}(a_{0},a_{1},\cdots,a_{k}):=\{z\in\hat{K}:~\mbox{for
every}~i\in\{0,1,\cdots,k\},~~\mbox{we have}~\iota(z)_{i}=a_{i}\}.
$$
Based on Statements $(a),(b)$ and our choice of $\phi,U_{0}$, there
is a constant $\hat{C}>0$ such that for every integer $k\geq1$ and
every point $z\in\hat{K}(\underbrace{0,\cdots,0}_{k})$, we have
\begin{equation}\label{equ_universal constant1}
    \hat{S}_{k}(\hat{G})(z)\geq
    k\hat{G}(x_{0})-\hat{C}.
\end{equation}
Meanwhile, due to \eqref{equ_nonsingular1} and
\eqref{equ_nonsingluar2},there is a sufficiently large constant
$\hat{C}$ such that
\begin{equation}\label{equ_universal constant2}
\hat{G}(z)\geq \hat{G}(x_{0})-\hat{C},~~\forall z\in U.
\end{equation}
Next we show that for every $k\geq 0$, every sequence
$a_{0},a_{1},\cdots a_{k}$ of elements of $\{0,1\}$ with $a_{0}=1$,
and every point $x\in\hat{K}(a_{0}a_{1}\cdots a_{k})$, then
\begin{equation}\label{equ_inequality}
    \hat{S}_{k+1}(\hat{G})(x)\geq(k+1)\hat{G}(x_{0})-2(a_{0}+a_{1}+\cdots+a_{k})\hat{C}.
\end{equation}
In fact, put $l:=a_{0}+\cdots+a_{k},~i_{l+1}:=k+1$, and also put
$0=i_{1}<i_{2}<\cdots i_{l}\leq k$ be all integers such that
$a_{i}=1$. It follows from \eqref{equ_universal constant1} and
\eqref{equ_universal constant2} that
$$
\hat{S}_{i_{j+1}-i_{j}}(\hat{G})(\hat{f}^{i_{j}}(x))\leq
(i_{j+1}-i_{j})(\hat{G})(x_{0})-2\hat{C},~~\forall
j\in\{1,\cdots,l\}.
$$
Summing over $j\in\{1,2,\cdots,l\}$, and we obtain desired
inequality \eqref{equ_inequality}. Thus, if we put
$$
\Phi(s):=\sum_{k=1}^{\infty}\exp(k\hat{G}(x_{0})-2\hat{C})s^{k},
$$
then each of the coefficients of
$$
\Upsilon(s):=\Phi(s)+\Phi(s)^{2}+\cdots
$$
is less than or equal to the corresponding coefficients of $\Xi$,
and therefore the radius of convergence of $\Xi$ is less than or
equal to that of $\Upsilon$. Since $\Phi(s)\to\infty$ as
$s\to\exp(-\hat{G}(x_{0}))^{-}$, there is an
$s_{0}\in(0,\exp(-\hat{G}(x_{0})))$ such that $\Phi(s_{0})>1$. It
follows that the radius of convergence of $\Upsilon$, and hence that
of $\Xi$, is less than or equal to $s_{0}$, and therefore it is
strictly less than $\exp(-\hat{G}(x_{0}))$. The proof of this lemma
is thus completed.
\end{proof}

\medskip

Once Lemma \ref{lem_period}, and Proposition \ref{prop_2new} are
proved, we follow the same strategy as in \cite{LiRiv13a} and
\cite{InoRiv12} to deduce Key Lemma. We include the proof for
completeness.

\subsection{Proof of Key Lemma}
\begin{proof}[Proof of Key Lemma]
If an ergodic measure $\mu\in \cM(f,J(f))$ is supported on a
repelling periodic point, the desired inequality follows from Lemma
\ref{lem_period}.

Otherwise, the ergodic measure $\mu\in\cM(f,J(f))$ is not supported
on a periodic orbit. By the topological exactness on $J(f)$, and the
ergodicity of $\mu$, it follows that $\mu$ is non-atomic and fully
supported on $J(f)$. By Proposition \ref{prop_2new}, for each $t<0$,
there exist a constant $C>0$, a connected and compact subset $B_{0}$
of $J(f)$, and a free IMFS $(\phi_{k})_{k=1}^{+\infty}$ generated by
$f$ with a time sequence $(m_{k})_{k=1}^{+\infty}$ that is defined
on $B_{0},$ such that for every $k\geq1$ and every point
$y\in\phi_{k}(B_{0})$, we have
\begin{equation}\label{equ_sum2}
S_{m_{k}}(-tG)((y))\geq m_{k}\int -tGd\mu-C.
\end{equation}
Since the IMFS $(\phi_{l})_{l=1}^{+\infty}$ is free, there is a
point $x_{0}\in B_{0}$ such that for every $\underline{l},
\underline{l}'\in \Omega^{*}$ with
$m_{\underline{l}}=m_{\underline{l}'}$, the sets
$\phi_{\underline{l}}(x_{0})$ and $\phi_{\underline{l}'}(x_{0})$ are
disjoint. Moreover, for every integer $k\geq1,$ every
$\underline{l}:=l_{1}\cdots l_{k}\in\Omega^{*}$, every
$y_{0}\in\phi_{\underline{l}}(x_{0})$, and every $j\in
\{1,\cdots,k-1\},$ the point
$$
y_{j}:=f^{m_{l_{1}}+m_{l_{2}}+\cdots+m_{l_{j}}}(y_{0})
$$
is in $\phi_{m_{j+1}}(B_{0})$. Hence, followed by \eqref{equ_sum2},
given a $t<0$, we have
\begin{align*}
&
S_{m_{\underline{l}}}(-tG)(y_{0})=S_{m_{l_1}}(-tG)(y_{0})+S_{m_{l_2}}(-tG)(y_{1})+\cdots+S_{m_{l_{k}}}(-tG)(y_{k})\\
&\geq \sum_{i=1}^{k}\left(m_{l_{i}}\int
-tGd\mu-C\right)=m_{\underline{l}}\int -tGd\mu-kC.
\end{align*}
This implies that for every $\underline{l}\in\Omega^{*}$, and every
$y_{0}\in\phi_{\underline{l}}(x_{0})$, we have
\begin{equation}\label{equ_series}
  \exp(S_{m_{\underline{l}}}(-tG)(y_{0}))\geq\exp(m_{\underline{l}}\int -tGd\mu)\exp(-|\underline{l}|C).
\end{equation}
In addition, we put
$$
\Xi_{n}:=\bigcup_{\underline{l}\in\Omega^{*},m_{\underline{l}}=n}\phi_{\underline{l}}(x_{0}),~~\forall
n\geq1,
$$
then the radius of convergence of the series
$$
\Xi(s):=\sum_{n=1}^{\infty}\left(\sum_{y\in\Xi_{n}}\exp(S_{n}(-tG))(y)\right)s^{n},
$$
is given by
$$
R:=\left(\limsup\limits_{n\to\infty}\left(\sum_{y\in\Xi_{n}}\exp(S_{n}(-tG))(y)\right)^{1/n}\right)^{-1},
$$
and in particular,
$$
\exp\left(-\limsup_{n\to\infty}\frac{1}{n}\log\sum_{y\in
    f^{-n}(x_{0})}\exp(S_{n}-tG(y))\right)\leq R.
$$
Hence, to complete the proof of Key Lemma, it is sufficient to prove
$R<\exp(\int tGd\mu)$. Denote
$$
\Phi(s):=\sum_{l=1}^{\infty}\exp(-C)\exp\left(m_{l}\int
-tGd\mu\right)s^{m_{l}}
$$
Followed by \eqref{equ_series}, and the fact that the IMFS
$(\phi_{k})_{k=1}^{+\infty}$ is free, each of the coefficients of
the series
$$
\Upsilon(s):=\sum_{i=1}^{+\infty}(\Phi(s))^{i}=\sum_{n=1}^{+\infty}
\left(\sum_{\underline{l}\in\Omega^{*},m_{\underline{l}}=n}\exp\left(m_{\underline{l}}\int
-tGd\mu\right)\exp(-|\underline{l}|C)\right)s^{n},
$$
is less or equal to the corresponding coefficient of series $\Xi.$
So the radius of convergence of $\Xi$ is less or equal to that of
$\Upsilon$. Note also that
$$
\lim\limits_{s\to\exp(-\int -tGd\mu)^-}\Phi(s)=+\infty.
$$
Hence, there is an $s_{0}\in(0,\exp(-\int\psi d\mu))$ such that
$\Phi(s_{0})\leq1$ and thus we have
$$
R\leq s_{0}<\exp(\int tG d\mu),
$$
which implies \eqref{equ_hyperbolicity} and completes the proof of
Key Lemma.
\end{proof}

\bigskip

\section{Hyperbolicity and the existence of a non-atomic conformal
measure}\label{sec_hyperbolic and conformal}

In this section, we will use the Key lemma to prove the
Hyperbolicity and the existence of a conformal measure for the the
new potential $-tG,$ with $t<0$. Based on this, it allows us to use
Keller's results on showing the absence of a phase transition of the
hidden pressure function in next Section. We begin with the
definition of the terminology ``conformal measure.''

\medskip

Let $f:J(f)\to J(f)$ be a continuous interval map. Given a Borel
measurable function $g:J(f)\to[0,+\infty)$, a Borel probability
measure $\mu$ on $J(f)$ is \emph{g-conformal} for $f$, if for each
Borel set $A\subset J(f)$ on which $f$ is injective, we have
$$
\mu(f(A))=\int_{A}gd\mu.
$$

\medskip

The main result in this section is stated as follows.

\begin{prop}\label{prop_hyperbolic and conformal measure}
Let $f:J(f)\to J(f)$ be an interval map in $\sA$, and let $G:J(f)\to
\R\cup\{-\infty\}$ be the upper semi-continuous potential in $\cU$
given by Proposition \ref{lem_cohomology}. Then for each integer
$t<0$,
\begin{enumerate}
  \item $-tG$ is hyperbolic;
  \item there is a non-atomic $\exp(P(f,-tG)+tG)$-conformal measure for
$f,$ with the support equals to $J(f)$.
\end{enumerate}
\end{prop}

The proof of Proposition \ref{prop_hyperbolic and conformal measure}
will be at the end of this section and requires a few Lemmas.

\begin{lemm}[Positive Lyapunov exponent]\label{lem_positive on the Lyap}
Let $f:J(f)\to J(f)$ be an interval map in $\sA$, let
$G:J(f)\to\R\cup\{-\infty\}$ be the upper semi-continuous potential
in $\cU$ given by Proposition \ref{lem_cohomology}, and measure
$\mu\in\cM(f,J(f))$ be an equilibrium state of $-tG$ with $t<0$,
then $\chi_{\mu}>0$.
\end{lemm}
\begin{proof}
The proof is divided by two parts. In part 1, we show that
$\widetilde{P}(t)>0,~~\forall t<0$. In part 2, we use
\eqref{equ_cohomologyintegral} to show that each equilibrium state
of $-tG,$ with $t<0$ has positive Lyapunov exponent.

{\bf 1.} Since $f$ is topologically exact on $J(f)$, there are two
disjoint closed subset $A_{1},A_{2}\subset J(f)$ with non-empty
interior such that $f^{N_{1}}(A_{1})=f^{N}(A_{2})=J(f)$. Let
$N:=\max\{N_{1},N_{2}\}$, and thus $f^{N}$ has a 2-horse shoe
$\{A_{1},A_{2}\}$. Therefore, there is an $f^{N}$-invariant compact
set $K\subset J(f)$, such that $f|_{K}$ is semi-conjugate to the
full shifts on 2-symbols. Hence
\begin{equation}\label{equ_positiveentropy}
h_{top}(f)=\frac{1}{N}h_{top}(f^{N})\geq\frac{1}{N}\log2>0.
\end{equation}
On the other side, Since $f$ is continuously differentiable, and
every periodic point is hyperbolic repelling, it follows from
\cite[Proposition A.1]{Riv13a} that every the measure
$\eta\in\cM(f,J(f))$ has $\chi_{\eta}\geq0$. Put a measure
$\nu\in\cM(f,J(f))$ be a maximum entropy measure, i.e.,
$h_{\nu}(f)=h_{top}(f)$, then \eqref{equ_positiveentropy} yields
that constant 0 function satisfies the bounded range condition.
Followed by the arguments in \cite{HK82}(see also \cite[Theo
B]{LiRiv13b}), such maximum entropy measure $\nu$ is unique and
non-atomic, and thus
\begin{equation}\label{equ_positvehiddenpressure}
\widetilde{P}(t)\geq
h_{\nu}-t\chi_{\nu}=h_{top}(f)-t\chi_{\nu}\geq\frac{1}{N}\log2>0,~~\forall
t\leq 0.
\end{equation}

\medskip

{\bf 2.} Put a measure $\mu\in\cM(f,J(f))$ be an equilibrium state
of the new potential $G$. It is sufficient to consider ergodic
$\mu$, the general case will be followed by ergodic decomposition.
We distinguish two cases.
\begin{itemize}
  \item $\mu$ is atomic. Then the topological exactness of $f$ and
ergodicity of $\mu$ yield that $\mu$ must be supported on a periodic
orbit. On the other hand, note that every periodic orbit are
hyperbolic repelling. So we have $\chi_{\mu}>0.$
  \item $\mu$ is non-atomic. Then following from \eqref{equ_cohomologyintegral} and
  \eqref{equ_positvehiddenpressure},
we have
\begin{align*}
P(f,-tG)=&h_{\mu}-\int_{J(f)} tGd\mu\\
\geq&\widetilde{P}(f,-tG)=\widetilde{P}(t)>0,~~~\forall t<0.
\end{align*}
So either $h_{\mu}>0$ or $\int_{J(f)} Gd\mu>0$. If $h_{\mu}>0$, then
by Ruelle's inequality, we have $\chi_{\mu}>0$; else if $\int_{J(f)}
Gd\mu>0$, then $\chi_{\mu}=\int_{J(f)} Gd\mu>0$.
\end{itemize}
In both cases, we have $\chi_{\mu}>0$, as wanted.
\end{proof}

\medskip

\begin{lemm}\cite[Prop 0.4]{Zh15}\label{lem_tree pressure}
Let $f:J(f)\to J(f)$ be an interval map in $\sA$, and
$u:J(f)\to\R\cup\{-\infty\}$ be a upper semi-continuous potential in
$\cU$ with $\Lambda(u)$ the resulting singular set. If $u$ is
hyperbolic and non-exceptional for $f$, then for every periodic
point $x\in J(f)$, or every non-periodic point $x\in
J(f)\backslash\bigcup_{i=-\infty}^{\infty}f^{i}(\Lambda(u))$, we
have
\begin{equation}\label{equ_tree pressure}
P(f,u)=\lim_{n\to\infty}\frac{1}{n}\log\sum_{y=f^{-n}(x)}\exp(S_{n}(u)(y)).
\end{equation}
\end{lemm}

\medskip

The following lemma is a consequence of a general method on
construction of a conformal measure, which is usually known as the
``Patterson-Sullivan method''. This method is introduced
respectively by Sullivan \cite{Sul83} in the setting of complex
rational maps, and by Denker and Urbanski \cite{DU91b} in the setting
of real interval maps.
\begin{lemm}\cite[Prop 3.2]{LiRiv13b}\label{lem_conformal measure}
Let $f:I\to I$ be a continuous interval map, let $X$ be a compact
subset of $I$ that contains at least 2 points and satisfies
$f^{-1}(X)\subset X$, and let $u:X\to\R\cup\{-\infty\}$ be a upper
semi-continuous function in $\cU$. Assume that $f$ has no periodic
critical point in $X$, and that there is a point of $X$ and an
integer $N>1$ such that the number
$$
P_{x_{0}}:=\limsup_{n\to+\infty}\frac{1}{n}\log\sum_{y\in
f^{-n}(x_{0})}\exp(S_{n}(u)(y))
$$
satisfies $P_{x_{0}}>\sup_{X}\frac{1}{N}S_{N}(u)$. Then there is an
atom-free $\exp(P_{x_{0}}-u)-$conformal measure for $f$. If addition
$f$ is topologically exact on $X$, then the support of this
conformal measure is equal to $X$.
\end{lemm}
\begin{rema}
Actually, Lemma \ref{lem_conformal measure} is proved in
\cite{LiRiv13b} by assuming the potential $u$ to be H\"{o}lder
continuous, however, its proof works without change for $u$ being in
$\cU$. In fact, the hypothesis $u\in\cU$ implies the following property:
\begin{align*}
&\forall\mu\in\cM,~~\forall A\subset J(f),~\mbox{with}~\mu(\partial
A)=\mu(\partial f(A)),~\mbox{and}~\overline{A}\subset
I\backslash\{x|~f~\mbox{not open at}~x\}\\
\Rightarrow&~~\mu(f(A))=\int_{A}\exp(P(f,u)-u)d\mu.
\end{align*}
This property is fundamental for the validity of the
Patterson-Sullivan method.
\end{rema}

\medskip

We also need a simple lemma below, and include its short proof for
completeness.
\begin{lemm}\label{lem_uppersemi}
For each interval map $f:J(f)\to J(f)$ in $\sA$, and every upper
semi-continuous function $\phi: J(f)\to\R\cup\{-\infty\}$, we have
\begin{equation}\label{equ_topandprob}
    \limsup_{n\to\infty}\sup_{J(f)}\frac{1}{n}S_{n}(\phi)=
    \sup_{\nu\in\cM(J(f),f)}\int\phi d\nu.
\end{equation}
\end{lemm}
\begin{proof}
On one hand, for each $\nu\in\cM(f,J(f))$, and each $n>0$, we have
$$
\int S_{n}(\phi)d\nu=\sum_{k=0}^{n-1}\int\phi\circ
f^{k}d\nu=n\int\phi d\nu.
$$
So
$$
\sup_{\nu\in\cM(f,J(f))}\int\phi d\nu\leq
\limsup_{n\to\infty}sup_{J(f)}\frac{1}{n}S_{n}(\phi).
$$
On the other hand,  for each $n$, the upper semi-continuality of
$\frac{1}{n}S_{n}(\phi)$ and the compactness of $J(f)$ yields that
there is an $x_{n}\in X$, such that
$\frac{1}{n}S_{n}(\phi)(x_{n})=\sup_{J(f)}\frac{1}{n}S_{n}(\phi).$

Put $\nu_{n}:=\frac{1}{n}\sum_{j=0}^{n-1}\delta_{f^{j}(x_{0})}$, so
$\nu_{n}(\phi)=\frac{1}{n}S_{n}(\phi)=\sup_{J(f)}\frac{1}{n}S_{n}(\phi)$.
Choose a subsequence $(n_{m})_{m=1}^{+\infty}$ of positive integers
such that
$$
\lim_{n\to\infty}\frac{1}{n_{m}}S_{n_{m}}\phi(x_{n_{m}})=\limsup_{n\to\infty}\sup_{J(f)}\frac{1}{n}S_{n}(\phi).
$$
Taking a subsequence if necessary, we can further assume that
$(\nu_{n_{m}})$ converges in the $weak^{*}$ topology to a measure
$\nu\in\cM(J(f),f)$. Then we have
\begin{alignat*}{2}
\limsup_{n\to\infty}\sup_{J(f)}\frac{1}{n}S_{n}(\phi)
&=\lim_{m\to\infty}\frac{1}{n_{m}}S_{n_{m}}(\phi(x_{n_{m}}))\\
&=\lim_{m\to\infty}\int \phi d\nu_{n_m}\\
&\leq\int\phi d\nu ~~~~~&\text{(using the upper semi-continuity of
$\phi$)}\\
&\leq \sup_{\nu\in\cM(J(f),f)}\int\phi d\nu.
\end{alignat*}
Therefore, we obtain the desired assertion \eqref{equ_topandprob},
and complete the proof of this lemma.
\end{proof}

\medskip

We are ready to prove the Proposition \ref{prop_hyperbolic and
conformal measure}.

\begin{proof}[Proof of Proposition \ref{prop_hyperbolic and conformal measure}.]
{\bf 1.} This part deals with the proof of Statement $(a)$ on the
hyperbolicity of $G$. We first prove the following Claim.

{\bf Claim:} If $\mu\in\cM(J(f),f)$ is an equilibrium state, then
the measure-theoretic entropy $h_{\mu}$ is strictly positive.

We prove the Claim by contradiction. Suppose there is an equilibrium
state $\mu$ with $h_{\mu}=0$, then
\begin{equation}\label{equ_equilibrium state}
    \int G d\mu=P(f,G).
\end{equation}
Replacing by an ergodic component if possible, we can further assume
$\mu$ satisfying \eqref{equ_equilibrium state} is ergodic. Note also
from Lemma \ref{lem_positive on the Lyap} that $\chi_{\mu}>0$. We
distinguish two cases.
\begin{itemize}
  \item If $\mu$ is atomic, then $\mu$ supports on a periodic orbit
  $\cO_{\mu}$. Let $x$ be a periodic point on $\cO_{\mu}$;
  \item If $\mu$ is non-atomic, then $\mu$ supports on $J(f)$. Recall $X$ the $\mu-$full measure set in Statement (b) in Key
  Lemma, and let $x$ be a point inside the set $\left(J(f)\backslash\bigcup_{i=-\infty}^{+\infty}f^{i}(\Lambda(G))\right)\cap
  X$. Such $x$ exists, since the intersection is a $\mu-$full
  measure set.
\end{itemize}
In both cases, we have for each $t<0$,
\begin{alignat*}{2}
P(f,-tG)&=\limsup_{n\to\infty}\frac{1}{n}\log\sum_{y=f^{-n}(x)}\exp(S_{n}(-tG)(y))&~~~~~~~~~~~~~~~~~~~~~\text{(Using
Lemma \ref{lem_tree pressure})}\\
&>\int -tG d\mu&~~~~~~~~\text{(Using $\chi_{\mu}>0$ and Key Lemma)}\\
&=P(f,-tG)&~~~~~~~~~~\text{(By \eqref{equ_equilibrium state})}.
\end{alignat*}
This is evidently impossible, so we obtain the Claim.

Hence,
\begin{alignat*}{2}
P(f,-tG)&>\sup_{\nu\in\cM(J(f),f)}\int -tG d\nu
&~~~~~~~~~~~~~~~~~\enskip&~~~~~~~~~~~~~~~~~~~~~~~~~~~~~~~\text{(Using the claim in particular $h_{\mu}>0$)}\\
&=
\limsup_{n\to\infty}\sup_{J(f)}\frac{1}{n}S_{n}(-tG)&\enskip&\text{(Using
Lemma \ref{lem_uppersemi})}.
\end{alignat*}
In other word, there exists an integer $n\geq0$ such that
$P(f,-tG)>\sup_{J(f)}\frac{1}{n}S_{n}(-tG)$, which means the
potential $-tG$ is hyperbolic.

\medskip

{\bf 2.} This part deals with Statement $(b)$ on the existence of a
conformal measure. Note that for every point $x_{0}\in
J(f)\backslash\bigcup_{i=-\infty}^{+\infty}f^{i}(\Lambda(G))$, we
have
\begin{alignat*}{2}
P_{x_{0}}&=P(f,-tG)&\medskip~~~~\text{(Using Lemma \ref{lem_tree pressure})}\\
&>\sup_{J(f)}\frac{1}{N}S_{N}(-tG)&\medskip~~~~~~~~~~~~~~~~~\text{(Using
hyperbolicity)}.
\end{alignat*}
By applying Lemma \ref{lem_conformal measure} with $X=J(f)$, there
is a non-atomic $\exp(P(f,-tG)+tG)$ conformal measure. Note also $f$
is topologically exact on $J(f)$, so the support of this conformal
measure is equal to $J(f)$. This proves the Statement $(b)$, and
thus completes the proof of Proposition \ref{prop_hyperbolic and
conformal measure}.
\end{proof}

\bigskip


\bigskip

\section{Makarov-Smirnov's formalism for interval
maps}\label{sec_makarov-smirnovformalism}In this section, we finish
the proof of the interval version of Makarov-Smirnov's formalism.
The proof relies on showing the spectral gap for the corresponding
transfer operator in a certain Keller's space. A good survey of this
methodology comes from Keller's original paper \cite{Kel85}, and
also from Rivera-Leterlier's lecture note \cite{RL15}.

The following lemma is given by \cite{Zh15}.
\begin{lemm}\cite[Lemm 0.6 and Prop 0.4]{Zh15}\label{lem_interation}
Let $f:J(f)\to J(f)$ be an interval map in $\sA$, and
$u:J(f)\to\R\cup\{-\infty\}$ be a upper semi-continuous potential in
$\cU$. For each $t<0,$ suppose $-tu$ is non-exceptional and is
hyperbolic with $N:=N(t)$ being the integer such that
$\sup_{J(f)}\frac{1}{N}S_{N}(-tu)<P(f,-tu)$. Let
$\widetilde{u}:=\frac{1}{N}S_{N}(u)$, the following
properties hold.
\begin{itemize}
  \item $\sup_{J(f)}-t\widetilde{u}<P(f,-t\widetilde{u})$,~~$\exp(-tu)$ and $\exp(-t\widetilde{u})$ are H\"{o}lder
  continuous with the same H\"{o}lder exponent;
  \item  $P(f,-tu)=P(f,-t\widetilde{u})$, $-tu$ and $-t\widetilde{u}$ share the
  same equilibrium states;
  \item $-t\widetilde{u}$ is non-exceptional;
  \item There is a non-atomic
  $\exp(P(f,-t\widetilde{u})+t\widetilde{u})$-conformal measure
  supported on $J(f)$.
\end{itemize}
\end{lemm}

\medskip

Before the proof the interval version of Makarov-Smirnov's
formalism, it is important to obtain the following.
\begin{prop}\label{prop_mar-smi_nonexceptional}
Let $f:J(f)\to J(f)$ be an interval map in $\sA$. Let $u:J(f)\to
\R\cup\{-\infty\}$ be an upper semi-continuous potential in $\cU$.
If $u$ is hyperbolic and non-exceptional, then for each $t<0$, the
following holds:
\begin{enumerate}
  \item there is a unique non-atomic equilibrium state $\mu$ of $f$
  for the potential $-tu$. Moreover, $\mu$ is fully supported on
  $J(f)$, and has strictly positive entropy. In addition, $\mu$ is
  exponentially mixing;
  \item the pressure function $P(f,-tu)$ is equal to $\widetilde{P}(f,-tu)$, and it is real analytic on
  $(-\infty,0)$.
\end{enumerate}
\end{prop}
\begin{proof}
Fix a negative value $t\in(-\infty,0)$, and let $N:=N(t)$ be the
integer such that $\sup_{J(f)}\frac{1}{N}S_{N}(-tu)<P(f,-tu)$. By
the hypothesis on $u\in \cU$, the function $\exp(-tu)$ is
$\alpha$-H\"{o}lder continuous. By the hypothesis of the
hyperbolicity and non-exceptionality, Proposition
\ref{prop_hyperbolic and conformal measure} gives that there exists
a $\exp(P(f,-tu)+tu)$-conformal measure $m$.

With this convention, put
$$
g_{t}:=\exp(-tu-P(f,-tu)).
$$
It is clear that $g_{t}$ is $\alpha$-H\"{o}lder continuous, so
$g_{t}$ is of bounded $1/\alpha$ variation. On the other hand, put
$g_{t,N}$ as in \eqref{equ_iternationss}, then
$\sup_{J(f)}g_{t,N}<1$. Therefore, for each
$\widetilde{\alpha}\in(0,\alpha]$, the assertions of Corollary
\ref{Lem_spectral gap} hold, with $p=1/\widetilde{\alpha}$. Let $A$
be the constant given by Corollary \ref{Lem_spectral gap}, and
consider the Keller's space $H^{\widetilde{\alpha},1}(m)$. In the
view of Corollary \ref{Lem_spectral gap}, it is sufficient to verify
the uniqueness of the equilibrium state in Statement $(a)$, and the
Statement $(b)$.

\medskip

We first proceed the proof on the uniqueness of the equilibrium
state. Put $\widetilde{u}:=\frac{1}{N}S_{N}(u)$, then equilibrium
state $\nu$ of $f$ for the potential $-t\widetilde{u}$ has
\begin{alignat*}{2}
h_{\nu}(f)&=P(f,-t\widetilde{u})+\int_{J(f)}t\widetilde{u}d\nu&\\
&\geq
P(f,-t\widetilde{u})-\sup_{J(f)}(-t\widetilde{u})>0&~~\text{(Using
Lemma \ref{lem_interation})}.
\end{alignat*}
By Ruelle's inequity, taking an ergodic exponent if necessarily, we
have $\chi_{\nu}>0$. Note also $f$ is topologically exact on $J(f)$,
then \cite[Theo 6]{Dob13b} yields that the equilibrium state $\nu$
of $f$ for the potential $-t\widetilde{u}$ is unique. Applying Lemma
\ref{lem_interation} again, we get the uniqueness of the equilibrium
state of $f$ for the potential $-tu$.

\medskip

In the rest of the proof, we will prove Statement $(b)$. In other
words, we will show the integer $N$ is actually independent of $t$
in a neighborhood of $t$. 
Fix a negative value $t$ and for each $\epsilon\in\R$, put
$$
u_{\epsilon}:=(-t+\epsilon)u,~~\mbox{and}~~g_{t+\epsilon}:=\exp(u_{\epsilon}-P(f,u_{\epsilon})).
$$
It is clear that when $\epsilon$ is sufficiently small, then
$u_{\epsilon}$ is in $\cU$, and notice the pressure is continuous,
so there is an $\epsilon_{0}>0$, such that
$$
\sup_{J(f)}\frac{1}{N}S_{N}(\epsilon
u)<P(f,u_{\epsilon})-\sup_{J(f)}\frac{1}{N}S_{N}(-tu),~~\forall
\epsilon\in (-\epsilon_{0},\epsilon_{0}).
$$
Therefore,
$$
\sup_{J(f)}\frac{1}{N}S_{N}(u_{\epsilon})\leq
\sup_{J(f)}\frac{1}{N}S_{N}(-tu)+\sup_{J(f)}\frac{1}{N}(\epsilon
u)<P(f,u_{\epsilon}).
$$
So, from Part $(1)$ of Corollary \ref{Lem_spectral gap}, we have
$\exp(P(f,u_{\epsilon}))$ is equal to the spectral radius
$\cL_{g_{t+\epsilon}}$. Moreover, from Part (3) of Corollary
\ref{Lem_spectral gap}, the function $\epsilon\to
\exp(P(f,u_{\epsilon}))$ is real analytic on
$(-\epsilon_{0},\epsilon_{0})$. The proof of this lemma is thus
completed.
\end{proof}

\bigskip

We are finally ready to proof the interval version of
Makarov-Smirnov's formalism.
\begin{proof}[Proof of Corollary \ref{coro_geometric}]
For each interval map $f:J(f)\to J(f)$ in $\sA,$ it directly follows
from the definitions of $P,\widetilde{P}$ that
$$
P(t)\geq\max\{\widetilde{P}(t),-t\chi_{\max}\},~~\forall t<0.
$$
On the other hand, Key Lemma implies that there exists a new
potential $G,-h\in\cU$ such that $-tG=-t\log|Df|-t(h\circ f)+th$ and
$h$ only has poles in $\Sigma_{\max}$, and such that
\begin{equation}\label{equ_final}
\widetilde{P}(t)=\widetilde{P}(f,-tG),~~\forall t<0.
\end{equation}
So
$$
P(t)=\max\{\widetilde{P}(t),-t\chi_{\max}\},~~\forall t<0.
$$
Note also from Proposition \ref{prop_hyperbolic and conformal
measure}, we have $-tG$ is hyperbolic and non-exceptional for $f$.
Applying Proposition \ref{prop_mar-smi_nonexceptional} with $u=G$,
we have
$$
\widetilde{P}(t)=\widetilde{P}(f,-tG)=P(f,-tG),~~\forall t<0,
$$
and $P(f,-tG)$ is real analytic at $(-\infty,0)$. Therefore we
obtain Statement $(a).$

\medskip

A phase transition occurs if and only if
$\chi_{*}>\widetilde{P}'(-\infty)$. Since we obtain the spectral gap
result in Keller's space for the new potential $-tG$ for every $t<0$, followed by the same arguments in \cite[Rem
3.7]{MakSmi00}, we have
\begin{alignat*}{2}
\widetilde{P}'(-\infty)&=\widetilde{P'}(f,\infty\cdot G)\\
&=P'(f,\infty\cdot G)\\
&=\sup\big\{\int_{J(f)}Gd\nu:~~\nu\in\cM(f,J(f))\big\}\\
&=\sup\big\{\int_{J(f)}Gd\nu:~~\nu\in\cM(f,J(f)),~~\nu(\Sigma_{\max})=0\big\}\\
&=\sup\{\chi_{\nu}:~\nu\in\cM(f,J(f)),~~\nu(\Sigma_{\max})=0\}.
\end{alignat*}
So we obtain Statement $(b)$.
\end{proof}

\bigskip

\section{Revisit Ruelle's result and complex Makarov-Smirnov's formalism}\label{sec_further researches}
The new ideas in our results yield also some progresses in the
complex setting. In this section, we will revisit Ruelle's result
and the complex version Makarov-Smirnov's formalism, and their
relations under the views of our Key Lemma \footnote{We highlight
that our Key Lemma also works well for all rational maps under an
even simpler proof.}. To simplify the notation, let $f$ be a
rational map of degree at least two on the Riemann sphere, then the
Julia set $J(f)$ is a perfect set (i.e., closed set without isolated
points) that is equal to the closure of repelling periodic points.
Moreover, $J(f)$ is completely invariant and $f$ is topological
exact on $J(f)$. Denote by $\cU$ an analogous subclass (but in the
complex setting) of upper semi-continuous potentials from
$J(f)\to\hat{\C}$ as in \eqref{equ_uppersemicontinuous1}, and denote
by $BV_{2}$ the Banach space of functions from $\C$ to itself, for
which the second derivatives are complex measures. Analogous to the
discussion in the real setting, we also restrict the action of $f$
on its Julia set throughout this section. By abusing the notation,
we also say $u:J(f)\to \C$ in $BV_{2}$ if it is a restriction of a
function $u'|_{J(f)}$, where $u'\in BV_{2}$.

\medskip

Recall that the proof on the Makarov-Smirnov's formalism for complex
rational maps is closely related to the following fact (stated
below) by Ruelle \cite{Rue92}.

\begin{fact}\cite[Coro 6.3]{Rue92}\label{fac_Ruelle's result}
If $\log|u|: J(f)\to\hat{\C}$ is a hyperbolic upper semi-continuous
potential with $u:J(f)\to\C$ in the functional space $BV_{2}$, and
satisfying a certain \emph{integrability condition} stated in
\cite[Coro 6.3]{Rue92} (This condition implies that $u$ vanishes at
all the critical points of $f$). Then the corresponding transfer
operator $\cL_{|u|}$ is bounded and has a spectrum gap under the
space of $BV_{2}$.
\end{fact}
Based our Key Lemma and the above Fact \ref{fac_Ruelle's result}, we
can reprove the complex Makarov-Smirnov's result under an addition
assumption of non-exceptionality.
\begin{coro}\label{coro_complex}
If $\log|Df|$ is non-exceptional, then the transfer operator $\cL$
associated by the weight $|Df|^{-t}$ admits a spectral gap under
$BV_{2}$ for each $t<0$. Moreover, $P(t)$ equals to
$\widetilde{P}(t)$, and is real analytic on $(-\infty,0)$, and
admits a unique non-atomic equilibrium state for each $t<0$.
\end{coro}
\begin{proof}
Using Key Lemma and Proposition \ref{prop_hyperbolic and conformal
measure}, the non-exceptionality hypothesis implies that
$-t\log|Df|$ is hyperbolic for every $t<0$. On the other side,
the weight
$|Df|^{-t}$ actually satisfies the integrability condition, see for example \cite[\S 1]{MakSmi00}.
Therefore, Fact \ref{fac_Ruelle's result} directly yields the
statement of Corollary \ref{coro_complex}.
\end{proof}
\begin{rema}
Compared to the proof of \cite[Theo 3.1]{MakSmi00}, this is a
simpler proof in particular on the validity of the uniqueness of the
equilibrium state. In fact, due to the hyerbolicity, the uniqueness
is a directly consequence of the standard arguments from
\cite{HR92,HK82,Prz90}.
\end{rema}

\medskip

However, the methods in the proof of Corollary \ref{coro_complex}
seem not sufficiently strong to deal with exceptional rational maps.
Let us explain the obstacles more precisely as follows. Given a
rational map $f$ on the Riemann space, denote by $k(c)$ the
multiplicity of its critical point at $c$. Following from the
discussion in \cite[\S4]{MakSmi00} (or \S\ref{sec_exceptional}),
there exist a lower semi-continuous function $h$, which only has log
poles in $\Sigma_{\max}$, and a number $\widetilde{\kappa}>0$ with
$$
\frac{\widetilde{\kappa}}{1-\widetilde{\kappa}}:=\min\{k(c):c\in
f^{-1}(a)\cap\mbox{Crit}(f),~a\in Per(f)\cap\Sigma_{\max}\},
$$
so that for each $t<0$, the weight $|Df|^{-t}$ can be replaced with
a homologous weight
$$
G_{\widetilde{\kappa},t}:=|Df|^{-t}\left(\frac{h\circ
f}{h}\right)^{\widetilde{\kappa} t},
$$
and meanwhile $\log G_{\widetilde{\kappa},t}$ belongs to $\cU$ and
is non-exceptional. Therefore, our Key Lemma yields that $\log
G_{\widetilde{\kappa},t}$ is also hyperbolic.

Unfortunately, the potential $\log G_{\widetilde{\kappa},t}$ is
beyond the hypothesis of Fact \ref{fac_Ruelle's result}. We state
the reasons as follows.  Following from the construction, the new
potential $G_{\widetilde{\kappa},t}$ will not vanish at least one
critical point. This means the integrability condition in Fact
\ref{fac_Ruelle's result} is never satisfied, thus one cannot apply
Fact \ref{fac_Ruelle's result} to $G_{\widetilde{\kappa},t}$ to get
a spectral gap.

As a comparison, recall that we use Lemma \ref{Lem_spectral gap} on
the Keller's space to create spectral gap for the interval maps, and
the hypothesis of Lemma \ref{Lem_spectral gap} don't require an
analogous integrability condition. So we don't have this obstacles
in the real setting. This is the fundamental difference between real
and complex setting, which illustrates the power of Keller's spaces.

The above discussions naturally lead to the following problem.

\begin{prob}\label{prob_2}
Given a rational map $f$ on the Riemann sphere, and a hyperbolic
upper semi-continuous potential $\log |u|$ in $\cU$, find a Banach
space on which the transfer operator $\cL_{|u|}$ acts with a
spectral gap.
\end{prob}

\medskip

The difficulty will be of course in the situation where $u$ does not
vanish at every critical point of $f$, and it is plausible to expect
such Banach spaces have analogous properties to the Keller's spaces
for complex situation. This space seems to be a generalization of
$BV_{2}$, and contains some Sobolev spaces $W_{1,p}(\hat{\C})$ with
the number $p>2$, and sufficiently closed to 2. For example, it
might be the space of complex functions on the Riemann sphere of
$p$-bounded variation.

\medskip

Positive outcomes of Problem \ref{prob_2} will yield a complex
version of Theorem \ref{theo_pressurefunction} and Theorem
\ref{theo_phasetransition}. What is more important, it will provide
a new (and perhaps simpler) proof of the orginal Makarov-Smirnov's
formalism for rational maps. To be more precise, the method deduced
from Problem \ref{prob_2} will directly prove that the transfer
operator $\cL_{G_{\widetilde{\kappa},t}}$ has a spectral gap. In
comparison, Makarov and Smirnov in \cite{MakSmi00} use the spectrums
of a sequence of transfer operators $\{\cL_{G_{\kappa,t}}\}$ with
$\kappa<\widetilde{\kappa}$ to approximate the spectrum of the
transfer operator $\cL_{G_{\widetilde{\kappa},t}}$. Actually, these
authors need to consider $\cL_{G_{\kappa,t}}$ acting on the Sobolev
spaces instead of applying Fact \ref{fac_Ruelle's result}, since
every potential $\log G_{\kappa,t}$ is no longer hyperbolic
(although every $G_{\kappa,t}$ satisfies the integrability
condition).

%
%

%
%


\section{Appendix A. Keller's space and spectral gaps}\label{app_keller's space}
This appendix provides some basic ideas/background from \cite{Kel85}
and \cite{LiRiv13b} on the Keller's space which are required in the
proof of Proposition \ref{prop_mar-smi_nonexceptional}. 

%
%
%
%

\subsection{Keller's space} Let $X$ be a
compact subset of $\R$ and $m$ be a Borel non-atomic probability
measure on $X$. We consider the quotient space on the space of
complex valued functions taking values on $X$, defined by agreement
on a set of full measure with respect to $m$.

Denote by $d$ the
pseudo-distance on $X$ defined by
$$
d(x,y):=m(\{z\in X:~x\leq z\leq y~\mbox{or}~y\leq z\leq x\}).
$$
Note for every $x\in X$ and every $\epsilon>0$, the set of ball
$$
B(x,\epsilon):=\{y\in X, d(x,y)<\epsilon\}
$$
has positive measure with respect to $m$.

Given a measurable function $h:X\to\C$ and $\epsilon>0$, for each $x\in X$, let
$$
\mbox{osc}(h,\epsilon,x):=\mbox{ess-sup}\{|h(y')-h(y)|:y,y'\in
B(x,\epsilon)\}
$$
and
$$
\mbox{osc}_{1}(h,\epsilon):=\int_{X}\mbox{osc}(h,x,\epsilon)dm(x).
$$
Given $A>0$, and for each $\alpha\in(0,1]$ and each $h:X\to \C$, put
\begin{equation}\label{equ_norm of keller's space}
\mbox{var}_{\alpha,1}(h):=\sup\limits_{\epsilon\in(0,A]}\frac{\mbox{osc}_{1}(h,\epsilon)}{\epsilon^{\alpha}}~~\mbox{and}~~
||h||_{\alpha,1}:=||h||_{1}+\mbox{var}_{\alpha,1}(h).
\end{equation}
Let
\begin{equation}\label{equ_keller's space}
H^{\alpha,1}(m):=\{\mbox{m-equivalence class of functions}~h:X\to\C,
||h||_{\alpha,1}<+\infty\}.
\end{equation}
We remark here $\mbox{var}_{\alpha,1}(h)$ and $||h||_{\alpha,1}$
only depend on the equivalence class of $h$, and
$(H^{\alpha,1},||\cdot||_{\alpha,1})$ is a Banach space.

\medskip

\subsection{Transfer operator}
Fix $p\geq1$. A function $h:X\to \C$ is of \emph{bounded
p-variation}, if
\begin{equation*}\label{equ_p-variation}
    \sup\left\{\left\{\sum_{i=1}^{k}|h(x_{i})-h(x_{i-1})|^{p}\right\}^{\frac{1}{p}}:k\leq1,~~x_{0}<\cdots<x_{k}\in
    X\right\}<+\infty.
\end{equation*}

Given $g:X\to [0,+\infty)$ to be a function of bounded
$p-$variation, let the \emph{transfer operator} $\cL_{g}$ be the
operator acting on the space
$$
\mbox{Eb}(X):=\{h:X\to \C,~|h|<+\infty\},
$$
according to the formula
\begin{equation}\label{equ_transfer operator}
    \cL_{g}(h)(x):=\sum_{y\in f^{-1}(x)}g(y)h(y).
\end{equation}
Such $g$ is also called the \emph{weight function}, or simply
\emph{weight}.

\medskip

\subsection{Spectral gap theorem}
Given an interval map $f\in\sA$, if the potential $\log g$ in $\cU$
is hyperbolic, then Keller's spaces are appropriate for the
corresponding transfer operation $\cL_{g}$ on which it admits a
quasi-compactness property. We state below more explicitly.

\medskip

\begin{lemm}\cite[Coro4.4]{LiRiv13b}\label{Lem_spectral gap}
Let $f:J(f)\to J(f)$ be an interval map in $\sA$, $g:J(f)\to
[0,+\infty)$ be a weight function of bounded $p$-variation, and
$\cL_{g}$ be the transfer operator defined in \eqref{equ_transfer
operator}. Suppose that there is an integer $n\geq1$ such that the
function
\begin{equation}\label{equ_iternationss}
g_{n}(x):=g(x)\cdots\cdots g(T^{n-1}(x))
\end{equation}
satisfies $\sup_{J(f)}g_{n}<1$. Suppose also $f$ admits a $g^{-1}$-
conformal measure $m$. Then we have the follow properties:
\begin{enumerate}
  \item The number 1 is an eigenvalue of $\cL_{g}$ of algebraic
  multiplicity 1. Moreover, there are constant $A>0$ and $\rho\in (0,1)$, such that the
  spectrum of $\cL_{g}|_{H^{1/p,1}(m)}$ is contained in
  $B(0,\rho)\cup\{1\}$;
  \item There exists a unique equilibrium state $\mu$ of $f$ for the
  potential $\log g$, and moreover $\mu\ll m$.
  \item There is a constant $C>0$, such that for every bounded
  measurable function $\varphi:X\to\C$, and every function $\psi\in
  H^{1/p,1}(m)$, the equilibrium state $\mu$ has the decay of correlation
  $$
C_{n}(\varphi,\psi)\leq
C||\varphi||_{\infty}||\psi||_{1/p,1}\rho^{n},~~\forall n\geq1;
  $$
  \item Given $\psi\in H^{1/p,1}(m)$, for each $\tau\in\C$ the
  operator $\cL_{\tau}$ defined by
  $$
\cL_{\tau}(h):=\cL_{g}(\exp(\tau\psi)\cdot h)
  $$
  is invariant under $H^{1/p,1}(m)$, and the restriction
  $\cL_{\tau}|_{H^{1/p,1}}(m)$ is bounded. Moreover, the map
  $\tau\mapsto\cL_{\tau}|_{H^{1/p,1}(m)}$ is real analytic in the
  sense of Kato on $\C$, and the spectral radius of
  $\L_{\tau}|_{H^{1/p,1}(m)}$ depends on a real analytic way on
  $\tau$ on a neighborhood of $\tau=0.$
\end{enumerate}
\end{lemm}

\medskip

\begin{small}

    \noindent \textsc{Yiwei Zhang} \quad \email{\href{mailto:yzhang@mat.puc.cl}{yzhang@mat.puc.cl}}

    \smallskip

    \noindent \textsc{Facultad de Matem\'aticas, Pontificia Universidad Cat\'olica de Chile. Avenida Vicu\~na Mackenna 4860, Santiago, Chile}

\end{small}

\end{document}